\documentclass[12pt]{amsart}

\usepackage{tikz-cd}

\usepackage[english]{babel}
\usepackage[utf8x]{inputenc}
\usepackage[T1]{fontenc}
\usepackage{amsmath}
\usepackage{amstext}
\usepackage{amsfonts}
\usepackage{amssymb}
\usepackage{amsthm}

\usepackage[alphabetic, initials]{amsrefs}
\usepackage{mathtools}
\usepackage{dsfont}
\usepackage{color}
\usepackage{graphicx}
\usepackage[colorinlistoftodos]{todonotes}
\usepackage[colorlinks, linkcolor=red, citecolor=blue, urlcolor=blue, hypertexnames=true]{hyperref}

\usepackage{float}
\usepackage[colorinlistoftodos]{todonotes}
\usepackage{theoremref}
\usepackage{enumitem}
\usepackage{dirtytalk}
\usepackage{todonotes}

\allowdisplaybreaks
\newtheorem{theorem}{Theorem}[section]

\newtheorem{prop}[theorem]{Proposition}
\newtheorem{cor}[theorem]{Corollary}
\newtheorem{thm}[theorem]{Theorem}
\newtheorem{lem}[theorem]{Lemma}

\newtheorem*{cor*}{Corollary}
\newtheorem*{conjecture*}{Conjecture}
\newtheorem*{thm*}{Theorem}
\newtheorem*{lem*}{Lemma}

\newtheorem*{prop*}{Proposition}
\theoremstyle{definition}

\newtheorem{observation}[theorem]{Observation}

\newtheorem*{defn*}{Definition}

\theoremstyle{remark}
\newtheorem{remark}[theorem]{Remark}

\title[]{Classification of Invariant Subalgebras in a class of factors with property (T)}

\author{Yongle Jiang*}

\address{Yongle Jiang, School of Mathematical Sciences, Dalian University of Technology,
Dalian, 116024, China}
\email{yonglejiang@dlut.edu.cn}

\author{Hongyi Li}
\address{Hongyi Li, School of Mathematical Sciences, Dalian University of Technology,
Dalian, 116024, China}
\email{hyli.math@gmail.com}

\thanks{*-Corresponding author}

\date{\today}

\begin{document}

\begin{abstract}
Let $n\geq 2$ and $G_n=\mathbb{Z}^n\rtimes SL_n(\mathbb{Z})$. We classify all $G_n$-invariant von Neumann subalgebras in $L(G_n)$. For $n=2$, this gives an alternative proof of the previous result of Jiang-Liu. For $n\geq 3$, this gives the first class of property (T) groups without the invariant subalgebras rigidity property but invariant subalgebras in the corresponding group factors can still be classified. As a corollary, $L(G_n)$ admits a unique maximal Haagerup $G_n$-invariant von Neumann subalgebra.
\end{abstract}



\maketitle

%\tableofcontents

\section{introduction}

Property (T), introduced in a  paper by Kazhdan \cite{kazhdan_T}, is a fundamental rigidity property for groups with far-reaching applications across several areas of mathematics, including group theory, dynamical systems, and operator algebras (see \cite{bdv} for an overview).
Its influence is especially evident in the proof of numerous rigidity theorems. For instance, in group theory, property (T) plays a key role in Margulis’s normal subgroup theorem for higher-rank lattices \cite{Margulis_book}. In operator algebras, it underpirds both Connes's pioneering work on factors with countable symmetry groups \cite{connes_T} and  Popa's striking works on cocycle and orbit equivalence superrigidity \cites{popa_T1, popa_T2, popa_Tc}. Moreover, it is central to many rigidity phenomena in the theory of von Neumann algebras, see \cites{popa_icm, vaes_icm, ioana_icm, houdayer_icm} for an overview.

Let $\Gamma$ be a countable discrete group—in particular, a higher-rank lattice subgroup in a semisimple Lie group $\mathbb{G}$ with Kazhdan’s property (T). In this paper, we study the classification of von Neumann subalgebras of the group von Neumann algebra $L(\Gamma)$ that are invariant under the conjugation action of $\Gamma$ following \cites{ab,kp,cd}.

One of the primary motivations of  \cites{ab,kp} was to seek a non-commutative generalization of Margulis's normal subgroup theorem. Let us now consider a different perspective for this line of research. We first note that recent years have witnessed increasing attention on the structure of property (T) II$_1$ factors such as $L(\Gamma)$. Currently, two lines of questions have been particularly prominent.

The first originates from Connes’ suggestion that there should be a rich analogy between the embedding $\Gamma<\mathcal{U}(L(\Gamma))$ into the unitary group of $L(\Gamma)$ and the embedding of a lattice $\Gamma<\mathbb{G}$  in the corresponding Lie group $\mathbb{G}$. In particular, it was expected that certain superrigidity phenomena in the operator algebra setting should hold (see \cites{jones_10,cp,peterson_online} for discussion on this). For advances in this direction, see e.g. \cites{bekka_invent, pt,cp,bh,bbhp,dp, dgghl} and reference therein. This line of inquiry essentially studies the position of $\Gamma$ inside the unitary group $\mathcal{U}(L(\Gamma))$.

The second major problem is Connes’ rigidity conjecture concerning isomorphisms of group von Neumann algebras. He conjectured that if $\Gamma$ is an infinite icc  (infinite conjugacy class) group with property (T), then any isomorphism $L(\Gamma)\cong L(\Lambda)$ for an arbitrary group $\Lambda$ forces $\Gamma\cong \Lambda$. A major result on this conjecture has been obtained recently in \cites{ cios}, where a property (T) group satisfying this conjectured property was constructed. Note that $L(\Gamma)\cong L(\Lambda)$ implies an isomorphism between their unitary groups. Moreover, when $\Gamma$ is icc, any such an isomorphism preserves $\mathbb{T}$ (viewed as constant unitaries) globally (since $\mathbb{T}=\mathcal{Z}(L(\Gamma))\cap \mathcal{U}(L(\Gamma))$), so it induces an isomorphism $\mathcal{U}(L(\Gamma))/\mathbb{T}\cong \mathcal{U}(L(\Lambda))/\mathbb{T}$. It is also clear that $\Gamma$ embeds into $\mathcal{U}(L(\Gamma))/\mathbb{T}\cong \text{Inn}(L(\Gamma))$, the group of inner automorphisms of $L(\Gamma)$.

Thus, both problems are related to studying the embedding
\begin{align*}
    \Gamma<\mathcal{U}(L(\Gamma))/\mathbb{T}\cong \text{Inn}(L(\Gamma)),
\end{align*}
i.e. the conjugation action of $\Gamma$ on $L(\Gamma)$. Consequently, investigating this conjugation action might  potentially help understanding these two celebrated questions.

In this context, inspired by \cite{kp}, a new specific form of rigidity—called the invariant subalgebra rigidity (ISR) property—was recently introduced and studied by Amrutam and the first named author in \cite{aj}. A countable discrete group $G$
is said to have the ISR property if every $G$-invariant von Neumann subalgebra in $L(G)$ is of the form $L(N)$ for some normal subgroup $N\lhd G$. In other words, the lattice of $G$-invariant subalgebras $L(G)$ becomes highly constrained and classifiable for $G$ with the ISR property.

While many groups, including certain higher-rank lattices \cite{kp}, acylindrically hyperbolic groups with trivial amenable radical \cite{cds}, finite direct sum of non-abelian free groups \cite{aj} and even broad classes of amenable groups \cites{jz, dj, adjs}, have been shown to satisfy ISR, much less is known about the structure of 
$G$-invariant subalgebras when 
$G$ does not have the ISR property. In fact, there are only two papers along this direction till now. In \cite{jiangliu}, together with Liu, the first named author classified all invariant von Neumann subalgebras in $L(\mathbb{Z}^2\rtimes SL_2(\mathbb{Z}))$. In the joint work with Amrutam-Dudko-Skalski \cite{adjs}, we construct an amenable groups $G$ without this ISR property but invariant von Neumann subalgebras in $L(G)$ are still classifiable.
But what happens for groups with property (T)? Can one classify all invariant von Neumann subalgebras in $L(G)$ for a group $G$ with property (T) but without the ISR property? Since it is believed that all non-amenable groups with trivial amenable radical may have the ISR property \cites{cds,dj}, it is nature to consider property (T) groups with nontrivial amenable radical in order to answer these questions. 

The purpose of this work is to  exhibit a natural class of property (T) groups that do not have the ISR property, yet for which a complete classification of 
$G$-invariant von Neumann subalgebras of 
$L(G)$ is still possible. The following is our main result, which answers \cite[Question 4.1]{jiangliu} completely.

\begin{thm}\label{thm}
Let $G_{n}=\mathbb{Z}^n\rtimes SL_n(\mathbb{Z}),\ n\ge 2$. Then a von Neumann subalgebra $P\subseteq L(G_{n})$ is $G_n$-invariant if and only if either $P=L(H)$ for some normal subgroup $H\subseteq G_n$ or $P=A_d$ for some $d \ge 1$, where $A_d:=\{x\in L(d\mathbb{Z}^n):~\tau(xu_g)=\tau(xu_{g^{-1}}), ~\forall~g\in d\mathbb{Z}^n\}\subsetneq L(d\mathbb{Z}^n)$, where $\tau$ denotes the canonical trace on $L(G_n)$ defined by $\tau(x)=\langle x\delta_e, \delta_e\rangle$ for any $x\in L(G_n)\subseteq B(\ell^2(G_n))$ and $d\mathbb{Z}^n=(d\mathbb{Z})^n$ is the subgroup of $\mathbb{Z}^n$. 
\end{thm}
Note that $G_n$ are higher rank lattice groups with property (T) for $n>2$ \cite[Example 1.7.4(i)]{bdv}.
It shows that inside $L(G_n)$, invariant von Neumann subalgebras only  arise from two natural sources, i.e. either from normal subgroups or from (measurable) factor maps  of the algebraic actions $SL_n(\mathbb{Z})\curvearrowright\widehat{d\mathbb{Z}^n}\cong \mathbb{T}^n$ inherent to the semi-direct product structure of $G_n$. We remark that a similar result in the amenable setting has been obtained in \cite{adjs}.

With this theorem at hand, we can prove the existence of a unique maximal Haagerup invariant von Neumann subalgebra as in \cite{jiangliu}.
\begin{cor}\label{cor}
Let $G_n=\mathbb{Z}^n\rtimes SL_n(\mathbb{Z})$. Then $L(\mathbb{Z}^n\rtimes \{\pm I_n\})$ for even $n$; respectively $L(\mathbb{Z}^n)$ for odd $n$ is the unique maximal Haagerup $G_n$-invariant von Neumann subalgebra in $L(G_n)$, where $I_n$ denotes the identity matrix in $SL_n(\mathbb{Z})$.
\end{cor}

\paragraph{\textbf{On the method of proof}}

Comparing our proof for $n>2$ with the proof given in \cite{jiangliu} for $n=2$, the key difference lies in the strategies used to classify non-amenable invariant subfactors, which is also the hard core for the whole proof.

In \cite{jiangliu}, the classification (for non-amenable invariant subalgebras) directly follows from   \cite[Theorem 5.1]{cds} since $G_2$ is an icc exact group which satisfies condition 2) in \cite[Theorem 5.1]{cds}. We remark that the proof of this theorem is largely based on techniques developed within Popa's powerful deformation/rigidity theory framework. It is not a surprise that this strategy is not applicable to $G_n$ for $n>2$.  Indeed, Chifan-Sinclair \cite[Theorem A]{cs} proved that if a group $G$ satisfies this condition, then $L(G)$ is solid in the sense of Ozawa \cite{ozawa_solid}, i.e. $A'\cap L(G)$ is amenable for every diffuse von Neumann subalgebra $A\subset L(G)$. It is clear that for $n>2$, if we take $A=L(\mathbb{Z})\subset L(\mathbb{Z}^n)\subset L(G_n)$, where $\mathbb{Z}=\langle e_1\rangle$ with $e_1=(1,0,\ldots, 0)^t\in \mathbb{Z}^n$, then $L(SL_2(\mathbb{Z}))\subset A'\cap L(G_n)$. Thus $L(G_n)$ is not solid.

To handle the abovementioned issue, we apply the character approach developed in \cites{jz,dj} while studying the ISR property. The new ingredient is Lemma \ref{3.2}, where we proved that $\mathbb{Z}^n$ satisfies a modified version of the so-called non-factorizable regular character property as introduced in \cite{dj}.

In fact, the above approach also works uniformly for all $n\geq 2$ (see Proposition \ref{prop: classify_invariant_subfactors}), although we need to split the whole proof according to the parity of $n$ due to the existence of non-trivial centers in $SL_n(\mathbb{Z})$ for even $n$. Even for $n=2$, this yields an alternative proof for the main theorem in the previous work \cite{jiangliu}. Indeed, instead of splitting the proof by considering amenable/non-amenable invariant subalgebras $P$ as in \cite{jiangliu}, we now proceed by first classifying the center $\mathcal{Z}(P)$ (see Lemma \ref{3.5}).

Finally, we note that although our proof makes no direct appeal to the property (T) of $G_n (n>2)$ or the relative property (T) for the pair $(G_2, \mathbb{Z}^2)$, the rigidity phenomena for the action $SL_n(\mathbb{Z})\curvearrowright \mathbb{T}^n$ that we do use, in particular, the classification of its ergodic measures \cite[Proposition 9]{burger} and of its factor actions \cite[Example 5.9]{wit}, are themselves indirect manifestations of rigidity for the ambient group or relative rigidity for the group-subgroup pair. Thus property (T) for $G_n (n>2)$ or the relative property (T) for the pair $(G_2, \mathbb{Z}^2)$ implicitly underpins our results. \\

\paragraph{\textbf{Organization of the paper}} The detailed plan of the article is as follows: after this introduction, in Section 2 we prepare lemmas on group aspects and dynamical properties related to $G_n$, and prove the non-factorizable property for $SL_n(\mathbb{Z})$-invariant characters on $\mathbb{Z}^n$ (Lemma \ref{3.2}). Besides the proof of Corollary \ref{cor}, the proof of Theorem \ref{thm} takes up the whole section 3. \\

\paragraph{\textbf{Acknowledgements}} 
J.Y. is grateful to Dr. Amrutam Tattwamasi for asking \cite[Question 4.1]{jiangliu} which started this project, also for taking time reading an early version of this paper and sending comments which help improving the presentation greatly. He also thanks Dr. Amrutam Tattwamasi, Prof. Artem Dudko and Prof. Adam Skalski for the collaboration \cite{adjs}, from which he has benefited a lot.
The work is partially supported by National Natural Science Foundation of China (Grant No. 12471118).

\section{Preliminaries}\label{section: preliminaries}

In this section, we prepare lemmas on group, character and dynamical properties related to $G_n=\mathbb{Z}^n\rtimes SL_n(\mathbb{Z})$.

\begin{lem}\label{3.1}
Let $G = SL_{n}(\mathbb{Z})$ for $n\ge 2$. Consider the natural group action $G=SL_n(\mathbb{Z})\curvearrowright\mathbb{T}^n$. Then for any given positive integer $k\ge1$, there exists at most countably many  points in $\mathbb{T}^{n}$ whose $G$-orbit contains $k$-many points.
\end{lem}
\begin{proof}
For any $s=(s_{ij})_{1\le i,j \le n}\in G$ and $z=(z_{1}, z_{2},\ldots, z_{n})^{t} \in \mathbb{T}^n$, the action is given by \begin{equation*}
z\overset{s}{\mapsto}sz=
\left(\begin{smallmatrix}
z_{1}^{s'_{11}}z_{2}^{s'_{12}}\cdots\ z_{n}^{s'_{1n}}\\
z_{1}^{s'_{21}}z_{2}^{s'_{22}}\cdots\ z_{n}^{s'_{2n}}\\
 \cdots\cdots\\
z_{1}^{s'_{n1}}z_{2}^{s'_{n2}}\cdots\ z_{n}^{s'_{nn}}
\end{smallmatrix}\right),
\end{equation*}
where $s'=(s^{-1})^{T}=(s'_{ij})_{1\le i,j \le n}$.

To avoid symbol conflicts, we directly use $\sqrt{-1}$ to represent the imaginary unit $i$. Write $z_{i}=e^{2\pi \sqrt{-1} \theta_{i}} $ for $i = 1,\dots,n$. Assume that the orbit of $z$ has size $k$, i.e. $\# \text{orb}(z)=\#\{gz:g\in G\}=k$. By the Orbit-Stabilizer Theorem, $[G:\text{stab}(z)]=k$, where $\text{stab}(z)=\{g\in G :gz=z\}$. Take any $s\in \text{stab}(z)$, the equality $sz=z$ translates to 
\begin{equation*}
\left(\begin{smallmatrix}
    e^{2\pi\sqrt{-1}(s'_{11}\theta_{1} +s'_{12}\theta_{2}+\cdots+s'_{1n}\theta_{n})}\\
     e^{2\pi\sqrt{-1}(s'_{21}\theta_{1} +s'_{22}\theta_{2}+\cdots+s'_{2n}\theta_{n})}\\
     \cdots\\
      e^{2\pi\sqrt{-1}(s'_{n1}\theta_{1} +s'_{n2}\theta_{2}+\cdots+s'_{nn}\theta_{n})}
\end{smallmatrix}\right)
=\left(\begin{smallmatrix}
    e^{2\pi\sqrt{-1}\theta_{1}}\\
     e^{2\pi\sqrt{-1}\theta_{2}}\\
     \cdots\\
      e^{2\pi\sqrt{-1}\theta_{n}}
\end{smallmatrix}\right),
\end{equation*}
which is equivalent to $\sum\limits_{j=1}^ns'_{ij}\theta_{j}=\theta_{j}+2k_i\pi,\forall i=1,\ldots,n$, where $k_{1},\ldots,k_{n}$ are integers. Thus the vector $\boldsymbol{\theta} = (\theta_{1},\dots,\theta_n)^t$ satisfies the linear system $((s^{-1})^T -I_{n})\boldsymbol{\theta} \in (2\pi\mathbb{Z})^n$. If there exists an $s\in \text{stab}(z)$ such that $\text{det}((s^{-1})^{T}-I_{n})\ne0$, then $\boldsymbol{\theta} = ((s^{-1})^T -I_{n})^{-1}\boldsymbol{b}$ for some interger vector $\boldsymbol{b}\in (2\pi\mathbb{Z})^n$. 

Clearly, to finish the proof, it suffices to check that $\{s\in G:\text{det}((s^{-1})^{T}-I_{n})=0\}$ could not contain any finite index subgroups of $G$. Note that $\text{det}((s^{-1})^{T}-I_{n})=0$ iff $\text{det}(s-I_{n})=0$.

Note that by \cite[Proposition 3.4]{jiang_jfa}, there exists some $s\in G$ such that the absolute values of all eigenvalues of $s$ are not equal to one. Hence, for any $M\geq 1$, the eigenvalues of  $s^M$ are not equal to one. Therefore, $\text{det}(s^M-I_{n})\neq 0$. Since we may take $M\geq 1$ such that $s^M$ lies in any given finite index subgroup, this finishes the proof.
\end{proof}

Recall that a character on a countable discrete group $G$ is a map $\phi: G\rightarrow \mathbb{C}$ such that $\phi$ is untial (i.e. $\phi(e)=1$), positive definite (see \cite[Definition 1.B.1]{bd}) and conjugate invariant (i.e. $\phi(sgs^{-1})=\phi(g)$ for all $s,g\in G$). Note that the conjugation  invariance condition for a character holds automatically on abelian groups $G$.
By Bochner's theorem, every unital positive definite function on a countable discrete abelian group $A$ corresponds to a probability measure on the Pontryagin dual $\widehat{A}$ \cite[Theorem D.2.2]{bdv}. 

The following lemma is the new ingredient needed to apply the character approach, which is inspired by \cite[Proposition 3.13]{dj}.

\begin{lem}\label{3.2}
Let $n\geq 2$. Let $\phi, \psi$ be two characters on $\mathbb{Z}^n$. Write $\phi = \int_{\mathbb{T}^n} d\mu$, $\psi = \int_{\mathbb{T}^{n}} d\nu$, where $\mu$ and $\nu$ are probability measures on $\mathbb{T}^n$. Assume that both $\mu$ and $\nu$ are $SL_{n}(\mathbb{Z})$-invariant and $\phi(g)\psi(g) = 0$ for all $e\ne g\in\mathbb{Z}^n$. Then $\phi \equiv \delta_{e}$ or $\psi \equiv \delta_{e}$, where $e$ denotes the neutral element in $\mathbb{Z}^n$. 
\end{lem}
\begin{proof}
By \cite[Proposition 9]{burger}, we know that both $\mu$ and $\nu$ are convex combinations of ergodic $SL_{n}(\mathbb{Z})$-invariant measures which are either the Haar measure or atomic measures supported on a finite $SL_{n}(\mathbb{Z})$-invariant subsets in $\mathbb{T}^n$.

Write $\mu=\lambda \cdot\text{Haar}\ +\sum\limits_{i=1}^{\infty}\lambda_{i}\mu _{i}\ ,\nu=\lambda' \cdot\text{Haar}\ +\sum \limits_{i=1}^{\infty}\lambda_{i}'\nu _{i}$, where all $\mu_{i}$ and $\nu_{i}$ denote some atomic $SL_{n}(\mathbb{Z})$-invariant measure on $\mathbb{T}^n$ and Haar denotes the Haar measure on $\mathbb{T}^n$. Note that we have a countable sum by Lemma \ref{3.1} and the fact that $1=\lambda \ +\sum\limits _{i=1}^{\infty}\lambda_{i}=\lambda' \ +\sum\limits _{i=1}^{\infty}\lambda_{i}'$ with all coefficients are non-negative.

Note that for any  $s\in \mathbb{Z}^n$, we have 
\begin{align*}
    \phi(s)=\lambda\delta_e(s)+\sum\limits_{i=1}^{\infty}\lambda_{i}\int_{\mathbb{T}^{n}}\langle\chi,s\rangle d\mu_{i}(\chi),
\end{align*}
\begin{align*}
\psi(s)=\lambda'\delta_e(s)+\sum\limits_{i=1}^{\infty}\lambda_{i}'\int_{\mathbb{T}^{n}}\langle\chi,s\rangle d\nu_{i}(\chi).
\end{align*}
In the above expression, if we write $s = (s_1,\dots,s_n)^t \in \mathbb{Z}^n$ and $\chi = (\chi_{1},\dots,\chi_n)^t \in \mathbb{T}^n$, then the pairing $\mathbb{T}^n\times \mathbb{Z}^n\overset{\langle -,-\rangle}{\longrightarrow}\mathbb{T}$ is defined by $\langle \chi,s\rangle:=\prod\limits_{i=1}^{n}\chi_{i}^{s_{i}}\in \mathbb{T}$.

We aim to show that either $\lambda=1$ or $\lambda'=1$. Assume this does not hold, then $0\le \lambda < 1 $ and $0\le \lambda'<1$.

Since $\mu_{i}$ is an atomic measure supported on some $SL_{n}(\mathbb{Z})$-invariant finite subset in $\mathbb{T}^n$, there exsits some finite index subgroup $G_{i}\in SL_{n}(\mathbb{Z})$ such that $\forall g\in G_{i}, gx =x$ for all $x \in \text{supp}(\mu_{i})$. Note that this means that $\forall v \in \mathbb{Z}^n$, we have $\langle gx,v \rangle = \langle x, v\rangle$, i.e. $\langle x,g^{-1}v -v \rangle = 1$ for all $v\in \mathbb{Z}^n$ and all $x\in \text{supp}(\mu_{i})$. Correspondingly, we denote by $F_{i}$ the finite index subgroup associated to $v_{i}$.

Write $c=\sum\limits_{i=1}^{\infty}\lambda_{i}, c'=\sum\limits_{i=1}^{\infty}\lambda'_{i}$. Since $0 <c \le 1$ and $0<c' \le 1$, we may find some $I$ large enough such that $\sum\limits_{i=1}^{I}\lambda_{i}>\frac{c}{2}$ and $\sum\limits_{i=1}^{I}\lambda'_{i}>\frac{c'}{2}$. Then by what we explained above, there exists some finite index subgroup $H$ of $SL_{n}(\mathbb{Z})$ such that $\forall g \in H$, we have $\langle x,g^{-1}v-v \rangle = 1\ ,\forall v\in \mathbb{Z}^{n},\ \forall x\in (\bigcup\limits_{i=1}^{I}\text{supp}(\mu _{i}))\bigcup(\bigcup\limits_{j=1}^{I}\text{supp}(\nu _{j}))\subset \mathbb{T}^n.$ Note that here, we have used the previous mentioned definition of pairing $\langle -,-\rangle.$
Specifically, we may take $H = \bigcap\limits_{i=1}^{I}(G_{i}\cap F_{i})$, which is still of finite index because a finite intersection of finite index subgroups has finite index. Now, for all $g\in H$ and all $v\in \mathbb{Z}^n$, we have \begin{align*}
    \int_{\mathbb{T}^n}\langle \chi,g^{-1}v-v \rangle d\mu_{i}(\chi)=1=\int_{\mathbb{T}^{n}}\langle \chi,g^{-1}v-v\rangle d\nu_{i}(\chi),1\le i\le I.
\end{align*}
Therefore, we deduce that for any $g\in H$ and any $v\in \mathbb{Z}^n$ with $s:=g^{-1}v-v\ne e$, then we have 
\begin{align*}
    |\phi(s)|&=|0+\sum\limits_{i=1}^{I}\lambda_{i}+\sum\limits_{i=I+1}^{\infty}\lambda_{i}\int_{\mathbb{T}^{n}}\langle\chi,s\rangle d\mu_{i}(\chi)|\\
    &\ge \sum\limits_{i=1}^{I}\lambda_{i}-\sum\limits_{i=I+1}^{\infty}\lambda_{i}=\sum\limits_{i=1}^{I}\lambda_{i}-(c-\sum\limits_{i=1}^{I}\lambda_{i})
    =2\sum\limits_{i=1}^{I}\lambda_{i}-c>0.
\end{align*}
Therefore, $\phi(s) \ne 0$. Similarly, $\psi(s) \ne 0$. This yields a contradiction to $\phi(s)\psi(s) = 0$.
\end{proof}

It is well-known that inside any countable discrete group $G$, there is a unique largest amenable normal subgroup, which is called the amenable radical of $G$, denote by $Rad(G)$. Note that $Rad(G)$ contains all amenable normal subgroups of $G$. We detemine the amenable radical of $G_n$ in the following lemma.

\begin{lem}\label{3.3}
Let $G_{n}=\mathbb{Z}^n\rtimes SL_n(\mathbb{Z}),n\ge 2$, 
\begin{equation*}
\text{Rad}(G_{n}) = 
\left\{
\begin{aligned}
&\mathbb{Z}^n ,  n~\text{is odd}\\
&\mathbb{Z}^n \rtimes \{\pm I_{n}\},\ n~\text{is even},
\end{aligned}
\right.
\end{equation*}
where $\text{Rad}(G_{n})$ is the amenable radical of $G_{n}$ and $I_{n}$ denotes the $n\times n$ identity matrix.
\end{lem}
\begin{proof}
Write $H=\text{Rad}(G_{n})$. Set $H\mathbb{Z}^n = \{hz\ |\ h\in H, z\in \mathbb{Z}^n \}$. Becasue both $H$ and $\mathbb{Z}^n$ are normal in $G_{n}$, their product $H\mathbb{Z}^n$ is also a normal subgroup in $G_{n}$.

Consider the short exact sequence:
\begin{equation*}
    1\longrightarrow \mathbb{Z}^n \longrightarrow H\mathbb{Z}^n \longrightarrow \frac{H\mathbb{Z}^n}{\mathbb{Z}^n} \cong \frac{H}{H\cap{\mathbb{Z}^n}} \longrightarrow 1,
\end{equation*}
Here $\mathbb{Z}^n$ is amenable. From the second isomorphism theorem, we obtain $\frac{H\mathbb{Z}^n}{\mathbb{Z}^n} \cong \frac{H}{H\cap{\mathbb{Z}^n}}$. The quotient group $\frac{H}{H\cap{\mathbb{Z}^n}}$ is amenable because $H$ is amenable. Since amenability is preserved under group extensions, it follows that $H\mathbb{Z}^n$ is also amenable. Then we obtain $H\mathbb{Z}^n \le H$ 
from maximality of the amenable radical. Together with $\mathbb{Z}^n\le H\mathbb{Z}^n$, we have the chain $\mathbb{Z}^n \le H\mathbb{Z}^n \le H$. It follows that $H= \mathbb{Z}^n \rtimes K $, where $K=\{g\in SL_{n}(\mathbb{Z})\ |\ (0,g)\in H\}$. Take any $g\in SL_{n}(\mathbb{Z}), k\in K$, we compute $(0,g)(0,k)(0,g)^{-1} = (0, gkg^{-1})$. Since $H\lhd G_n$, we have $(0, gkg^{-1})\in H$, and hence $ gkg^{-1} \in K$. Thus $K\lhd SL_{n}(\mathbb{Z})$. Moreover, $K \cong \frac{H}{\mathbb{Z}^n}$ and $H$ are amenable, so is $K$. 

It is well-known that the amenable radical of $SL_{n}(\mathbb{Z})$ is $\{\pm I_{n}\}$ for even $n$ and trivial for odd $n$. Indeed, for $n\geq 3$, we may apply Margulis's normal subgroup theorem (\cite[Chapter IV]{Margulis_book}) to deduce that the amenable radical is a finite normal subgroup in $SL_n(\mathbb{Z})$ and hence contained in the center of $SL_n(\mathbb{Z})$ (\cite[Section 17.1]{Morris_book}), therefore, it equals the center. For $n=2$, this is explained in the proof of \cite[Proposition 2.10]{jiangskalski}.

 Since $K$ is a normal amenable subgroup of $SL_{n}(\mathbb{Z})$, we have $K\subseteq \{\pm I_{n}\}$. Consequently, $H = \mathbb{Z}^n \rtimes K \subseteq \mathbb{Z}^n \rtimes \{\pm I_{n}\}$. It's straightforward to check $\mathbb{Z}^n \rtimes \{\pm I_{n}\}$ is normal in $G_{n}$ and amenable. By the maximality of $H$, we obtain $\mathbb{Z}^n \rtimes \{\pm I_{n}\} =H$. 
\end{proof}

Recall that in \cite{jiangskalski}, the concept of the Haagerup radical of a group, i.e. the largest normal Haagerup subgroup, was studied. Although it is still unclear whether every countable discrete group admits the Haagerup radical, it was shown that for $G_2=\mathbb{Z}^2\rtimes SL_2(\mathbb{Z})$, its Haagerup radical exists and coincides with its amenable radical in \cite[Proposition 2.10]{jiangskalski}, which was extended in \cite[Proposition 4.1]{valette} for other semi-direct product ambient groups.
We extend the above result on the Haagerup radical of $G_2$ to all $n\geq 2$, which will be needed for the proof of Corollary \ref{cor}.

\begin{lem}\label{lem: Haagerup radical in G_n}
Let $n\geq 2$ and $G_n=\mathbb{Z}^n\rtimes SL_n(\mathbb{Z})$. Then $G_n$ admits the Haagerup radical, which is $\mathbb{Z}^n$ for odd $n$ and $\mathbb{Z}^n\rtimes \{\pm I_n\}$ for even $n$, where $I_n$ denotes the $n\times n$ identity matrix.
\end{lem}
\begin{proof}
    For $n=2$, this was proved in \cite[Proposition 2.10]{jiangskalski}. We may assume that $n\geq 3$. Following the proof of \cite[Proposition 2.10]{jiangskalski}, let $H$ be any normal subgroup of $G_n$ with the Haagerup property and $U_n=\mathbb{Z}^n\rtimes \{\pm I_n\}$ for even $n$ and $U_n=\mathbb{Z}^n$ for odd $n$. It suffices to show that $H\subseteq U_n$. Clearly, $U_n$ is amenable and normal inside $G_n$. Thus $HU_n$ is a normal subgroup in $G_n$ with the Haagerup property. As $U_n\subseteq HU_n$, we deduce that $HU_n=\mathbb{Z}^n\rtimes K_n$ for some normal subgroup $K_n$ in $SL_n(\mathbb{Z})$ with $\{\pm I_n\}\subseteq K_n$ for even $n$. Note that $K_n$ also has Haggerup property. By Margulis's normal subgroup theorem, this implies that $K_n\subseteq C(SL_n(\mathbb{Z}))=\{\pm I_n\}$ for even $n$ or trivial for odd $n$. In other words, $HU_n\subseteq U_n$. Thus $H\subseteq HU_n=U_n$.
\end{proof}

The following should be well-known. Since we could not find an appropriate reference in the literature, we decide to include a proof. 
\begin{lem}\label{lem: invariant subgroups in Zn} Let $n\geq 2$. Let $H$ be any $SL_{n}(\mathbb{Z})$-invariant subgroup of $\mathbb{Z}^n$, then $H = d\mathbb{Z}^n$ for some $d\in \mathbb{Z}$. \end{lem}
\begin{proof}

Let $d:=min\{\text{gcd}(\left|h_1\right|,\ldots,\left|h_n\right|)~|~ (h_1,\ldots,h_n)^t \in H\}$. Since $\mathbb{N}\setminus\{0\}$ is bounded below, such a $d$ exsits. Hence there exsits $(a_1,\dots,a_n)^t\in H$ such that $\text{gcd}(a_1,\dots,a_n)=d$. By Bezout’s Identity, we get $n$ integers $x_1,\dots,x_n\in \mathbb{Z}$ such that $x_1a_1+\dots+x_na_n = d$.

We consider the first coordinate. Since $H$ is $SL_{n}(\mathbb{Z})$-invariant, we have $g_2(a_1,\dots,a_n)^t = (a_1+x_2a_2,a_2,\dots,a_n)^t \in H$, where $g_2= \left(\begin{smallmatrix}
     1 & x_2 \\
     0 & 1 \\
     & & I_{n-2}
   \end{smallmatrix}\right)\in SL_{n}(\mathbb{Z})$. Then \begin{equation*}
      \left(\begin{smallmatrix}
           a_1+x_2a_2\\
           a_2\\
           \vdots\\
           a_n
       \end{smallmatrix}\right)-\left(\begin{smallmatrix}
            a_1\\
           a_2\\
           \vdots\\
           a_n
\end{smallmatrix}\right)=\left(\begin{smallmatrix}
            x_2a_2\\
           0\\
           \vdots\\
           0
       \end{smallmatrix}\right)\in H
   \end{equation*}
Similarly, for any $i\ne 1$, we have $(a_ix_i,0,\dots,0)^t \in H$. If we take $g_1= \left(\begin{smallmatrix}
     1 & 0 \\
     x_1 & 1 \\
     & & I_{n-2}
     \end{smallmatrix}\right)$, then $g_1(a_1,\dots,a_n)^t-(a_1,\dots,a_n)^t=(0,x_1a_1,0,\dots,0)^t \in H$. Note that $H$ is a subgroup, its $SL_{n}(\mathbb{Z})$-invariance also implies $-SL_{n}(\mathbb{Z})$-invariance. Hence, swapping the first two rows yields $(x_1a_1,0,\dots,0)^t\in H$, and subsequently we have $(d,0,\dots,0)^t \in H$. Swapping the two rows yields $(0,\dots,0,d,0,\dots,0)^t\in H$, where $d$ is at the $i$-th position, which shows that $d\mathbb{Z}^n \subseteq H$.

Next, we consider the reverse inclusion. We prove that for any $x\not\in d\mathbb{Z}^n$, we have $x\not\in H$.

Suppose $x=(x_1,\dots,x_n)^t\not\in d\mathbb{Z}^n$, that is $\exists~ i, d\nmid x_i$, which shows that $x_i\equiv x_i'~\text{mod}~d$ for some $x_i'$ with $0<x_i'<d$. If $x\in H$, we have \begin{equation*}
    \left(\begin{smallmatrix}
        x_1\\
        \vdots\\
        x_i'\\
        \vdots\\
        x_n 
\end{smallmatrix}\right)=\left(\begin{smallmatrix}
         x_1\\
        \vdots\\
        x_i\\
        \vdots\\
        x_n 
    \end{smallmatrix}\right)-\frac{x_i-x_i'}{d}\left(\begin{smallmatrix}
         0\\
        \vdots\\
        d\\
        \vdots\\
        0
    \end{smallmatrix}\right)\in H.
\end{equation*}
Then we obtain $\text{gcd}(x_1,\dots,x_i',\dots,x_n)<d$, which contradicts the minimality in the definition of $d$.
\end{proof}

\begin{remark}
    We remark that the above lemma can also be proved by noticing the fact that for any $v=(a_1,\ldots, a_n)^t\in \mathbb{Z}^n$ with $1=\text{gcd}(\left|a_1\right|,\ldots, \left|a_n\right|)$, then there exsits some $g\in SL_n(\mathbb{Z})$ such that the first column of $g$ equals $v$.
\end{remark}

Using the above remark, it is not hard to see that the following lemma holds true. We decide to include a slightly different proof.

\begin{lem} \label{3.6}
 Let $n\geq 2$ and $G_n=\mathbb{Z}^n\rtimes SL_n(\mathbb{Z})$. Let $P\subseteq L(G_n)$ be any $G_n$-invariant von Neumann subalgebra. Denote by $E: L(G_n)\rightarrow P$ the trace $\tau$-preserving conditional expectation onto $P$. Write $s=-I_n$, where $I_n$ is the $n\times n$ identity matrix in $SL_n(\mathbb{Z})$. Assume that $E(s)=0$. Then $E(vs)=0$ for all $v\in \mathbb{Z}^n$ if and only if $ E(e_1s)=0$, where $e_1=(1, 0,\dots,0)^t\in\mathbb{Z}^n$.   
\end{lem}
\begin{proof}
The \say{only if} direction clearly holds true. We just need to check that the \say{if} direction holds. Assume that $E(e_{1}s) = 0$. We now proceed by mathematical induction on the number of nonzero coordinates of $v\in\mathbb{Z}^n$.

For any $v = (d,0,\dots,0)^t\in \mathbb{Z}^n$, we get that $0=vE(s)v^{-1}=E(vsv^{-1})=E(v\sigma_s(v^{-1})s)=E((2d, 0,\dots, 0)^ts)$; similarly, $0=vE(e_1s)v^{-1}=E((2d+1, 0,\dots, 0)^ts)$. Since $d\in\mathbb{Z}$ is arbitrary, we deduce that $E((x_{1},0,\dots,0)^t s)=0$ for all $x_{1}\in \mathbb{Z}$.

\textbf{Base case $k=1$.} Let $v$ has exactly one nonzero coordinate, i.e. $v = x_{1}e_{i}:=(0,\dots,0,x_{1},0,\dots,0)^t$ with $x_1$ in the $i$-th position, where $0\ne x_1\in \mathbb{Z}$ and $e_i:=(0,\dots,0,1,0,\dots,0)^t $ with 1 in the $i$-th position.
Obviously, we can find a matrix $g_{1}\in SL_{n}(\mathbb{Z})$ such that $g_{1}\cdot (x_1,0,\dots,0)^t = v$(where $g_{1}\cdot (x_1e_1)$ denotes the matrix left multiplication), and thus we have $E(vs) = E(g_{1}(x_1e_1)g_{1}^{-1}s) = g_{1}E(x_{1}e_1s)g^{-1} = 0$. Therefore, the statement holds for every vector with exactly one nonzero coordinate.

Assume for some $k\ge1$ that $E(ws) = 0$ for all $w\in \mathbb{Z}^n\setminus\{0\}$ with $||w||_{0}\le k$, where $||w||_{0}$ denotes the number of nonzero coordinates of $w$.

\textbf{Induction step.} Let $v = (x_1,x_2,\dots,x_n)^t$ satisfy $||v||_{0} = k+1$. Obviously, we can reorder the coordinates of $v$ by a suitable matrix $g_{k+1}\in SL_{n}(\mathbb{Z})$ so that all its nonzero entries appear in the first $k+1$ positions. Consequently, we may assume without loss of generality that $v=(x_1,x_2,\dots,x_{k+1},0,\dots,0)^t,\,x_i\ne0(1\le i \le k+1)$.\\
Consider the first two coordinates $x_1,x_2$ of $v$, set $q=\text{gcd}(|x_1|, |x_2|)$. By Bezout's Identity, we get two integers $a, b\in\mathbb{Z}$ such that $x_1a-x_2b=q$, then set $g:= 
   \left(\begin{smallmatrix}
     x_1/q & b \\
     x_2/q & a \\
     & & I_{n-2}
   \end{smallmatrix}\right)
   \in SL_{n}(\mathbb{Z})$.
So we have following equation:
\begin{equation*}
    g\cdot \left(\begin{smallmatrix}
        q\\
        0\\
        x_3\\
        \vdots\\
        x_{k+1}\\
        \vdots\\
        0
\end{smallmatrix}\right)=\left(\begin{smallmatrix}
     x_1/q & b \\
     x_2/q & a \\
     & & I_{n-2}
   \end{smallmatrix}\right) \left(\begin{smallmatrix}
        q\\
        0\\
        x_3\\
        \vdots\\
        x_{k+1}\\
        \vdots\\
        0
\end{smallmatrix}\right)=\left(\begin{smallmatrix}
        x_1\\
        x_2\\
        x_3\\
        \vdots\\
        x_{k+1}\\
        \vdots\\
        0
    \end{smallmatrix}\right).
\end{equation*}
Put $w:=(q,0,x_3,\dots,x_{k+1},0,\dots,0)^t$. Because the second coordinate of $w$ is 0, its nonzero entries are $q, x_{3},\dots,x_{k+1}$, hence $||w||_{0}\le k$, so we have $E(ws) = 0$. Then we obtain equation $E(vs) = E(gwg^{-1}s) = gE(ws)g^{-1} = 0$. Thus the statement holds for vectors with $k+1$ nonzero coordinates.

By mathematical induction, $E(vs) = 0 $ for every nonzero vector $v\in\mathbb{Z}^n$ with $||v||_{0} = k$ for any $1\le k \le n$. The case $v=0$ is trivial.
Consequently, $E(e_1s) = 0$ implies $E(vs)=0$ for all $v\in\mathbb{Z}^n$.
\end{proof}

\section{Proof of Theorem \ref{thm} and Corollary \ref{cor}}\label{section: proof}

We first present the classification of invariant subfactors via the character approach developed in \cite{dj,jz}, which works uniformly for all $n\geq 2$.
\begin{prop}\label{prop: classify_invariant_subfactors}
    Let $P$ be a $G_n$-invariant subfactor in $L(G_n)$, where $G_n=\mathbb{Z}^n\rtimes SL_n(\mathbb{Z})$ and $n\geq 2$. Then there exists some normal subgroup $N$ in $G$ such that $P=L(N)$.
\end{prop}
\begin{proof}
By the proof of \cite[Theorem 3.1]{cds}, there is some normal subgroup $N\lhd G_n$ such that $L(N)=P\bar\otimes (P'\cap L(N))$. Let $E_{1}:L(G_n)\twoheadrightarrow P,~ E_{2}:L(G_n)\twoheadrightarrow P'\cap L(N)$ be the two trace $\tau$-preserving conditional expectations. Two characters $\phi$ and $\psi$ on $G_n$ could be introduced by defining $\phi(g):=\tau(E_{1}(g)g^{-1}),\psi(g)=\tau(E_{2}(g)g^{-1}), \forall g\in G_n$. Then  from the calculation used in the proof of \cite[Proposition 3.2]{jz}, we know that $\phi(g)\psi(g)=0$ for all $e\ne g\in N$.

Note that $N\cap \mathbb{\mathbb{Z}}^n$ is an $SL_{n}(\mathbb{Z})$-invariant subgroup of $\mathbb{Z}^n$. We obtain that $N\cap \mathbb{\mathbb{Z}}^n = \{0\}$ or $d_1\mathbb{Z}^n$ for some non-zero $d_1\in \mathbb{Z}$ by Lemma \ref{lem: invariant subgroups in Zn}.

If $N\cap \mathbb{\mathbb{Z}}^n = \{0\}$, then  $N\subseteq C_{G_n}(\mathbb{Z}^n)=\{A\in G_n|Av=vA,\forall v\in \mathbb{Z}^n\}=\mathbb{Z}^n$. Therefore $N$ is abelian, it follows that $P\subseteq L(N)$ is abelian, then $P = \mathcal{Z}(P) = \mathbb{C}$.

If $N\cap \mathbb{\mathbb{Z}}^n = d_1\mathbb{Z}^n$ for some nonzero $d_1\in \mathbb{Z}$, then $\phi|_{d_{1}\mathbb{Z}^n}$ and $\psi|_{d_{1}\mathbb{Z}^n}$ are two $SL_{n}(\mathbb{Z})$-invariant characters on $d_1\mathbb{Z}^n$ with $\phi|_{d_{1}\mathbb{Z}^n}(g)\psi|_{d_{1}\mathbb{Z}^n}(g) = 0,~\forall e\ne g \in d_1\mathbb{Z}^n$. Using Lemma \ref{3.2}, we just need to consider two cases.

\textbf{Subcase 1.} $\phi|_{d_{1}\mathbb{Z}^n}=\delta_e$.

Then we just need to check that for all $e\neq g \in N$, then $\phi(g)=0$. Write  $g =vt= (v,t)$, where $v\in \mathbb{Z}^n$ and $t\in SL_n(\mathbb{Z})$. Without loss of generality, we may assume that $t\ne I_n$, the $n\times n$ identity matrix in $SL_n(\mathbb{Z})$.

To prove $\phi(g)=0$, we will pick infinitely many suitable $v_n\in d_1\mathbb{Z}^n$ to be determined later. Consider $g_{n}:=v_{n}gv_{n}^{-1}=(v_{n}+v-\sigma _{t}(v_{n}),t)\in N$. Here $\sigma_t(v_n)=t\cdot v_n$ is the matrix left multiplication.  Note that $g_{n}g_{m}^{-1}=(v_{n}-v_{m}+\sigma_t(v_{m})-\sigma_t(v_{n}),I_{n})\in d_1\mathbb{Z}^n$. Let us pick $v_n$ such that $g_ng_m^{-1} \ne e$; equivalently, we need $v_n - v_m \ne \sigma_t(v_n - v_m)$ for all $n\ne m$. For example, we may take $v_n:=nw$ for any $w\in d_1\mathbb{Z}^n$ with $\sigma_{t}(w)\ne w$ which is possible since $t\neq I_n$.
Then we get that $\phi(g_ng_m^{-1})=0$ and \cite[Lemma2.7]{dj} yields that $\phi(g) = \phi(g_n) = 0$. This show that $\phi|_{N} \equiv \delta_e$, then we have $E_1(g) = 0,~\forall g\neq e$ and hence $P=\mathbb{C}$.

\textbf{Subcase 2.} $\psi|_{d_{1}\mathbb{Z}^n} = \delta_e$.

This is similar to the above proof and we conclude that $P'\cap L(N) = \mathbb{C}$ and thus $L(N)=P$.
\end{proof}

\begin{proof}[Proof of Theorem \ref{thm}]
    The \say{if} direction in Theorem \ref{thm} is easily verified, hence we just need to prove the \say{only if} direction, which we record as  Proposition \ref{prop: version of prop 3.1} and Proposition \ref{prop: version of prop 3.2} below
due to the difference while classifying invariant subalgebras with non-trivial center depending on the parity of $n$.
\end{proof} 
We split this section into two subsections.

\subsection{The proof for odd $n$}

We first prepare one lemma needed for the proof of Proposition \ref{prop: version of prop 3.1}.

\begin{lem}\label{3.4}
Let $G_{n}=\mathbb{Z}^n \rtimes SL_n(\mathbb{Z}), n\ge 2$ and n is odd. Let $d\ge1$ and $A_{d}:=\{\sum_{v\in d\mathbb{Z}^{n}}c_{v}v:c_{v}=c_{-v}\in\mathbb{C},\forall v\in d\mathbb{Z}^{n}\}\subseteq L(d\mathbb{Z}^n)$. Then $A_d' \cap L(G_{n})=L(\mathbb{Z}^{n})$.
\end{lem}
\begin{proof}
The direction $\supseteq$ is immediate. We just need to check $\subseteq$ holds. For this, it suffices (say by \cite[Lemma 1.6]{packer}) to check the essential freeness for the quotient action $SL_{n}(\mathbb{Z})\curvearrowright\widehat{d\mathbb{Z}^{n}}/{\sim}$, where we write $A_d\cong L^{\infty}(\widehat{d\mathbb{Z}^n}/{\sim})$.

To check this, we may assume without loss of generality that $d = 1$. Note that $\sim$ is defined on $\mathbb{T}^n$ as follows. For any $z=(z_1,\ldots,z_n), w=(w_1,\ldots ,w_n)\in \mathbb{T}^n$,  $z\sim w$ iff $z_i=w_i$ for all $1\le i \le n$ or $z_iw_i=1$  for all $1\le i \le n$.

Take any $e\ne g\in SL_{n}(\mathbb{Z})$, we aim to show that for the Haar measure $\mu$ on $\mathbb{T}^n$, we have
\begin{align*}
        \mu(\{z\in \mathbb{T}^{n}:gz\sim z\})=0.
    \end{align*}
Write $z=(z_1,\ldots,z_n)\in \mathbb{T}^n$, where $z_{k} = e^{i\theta_{k}}$ and $\theta_k \in \mathbb{R}$ for $1\le k \le n$. Then 
note that $gz\sim z$ iff we have either one of the following system of equations hold for some $k_i\in \mathbb{Z}, 1\le i \le n$,
 \begin{align*}
        (A-I_{n})\begin{pmatrix}
            \theta_{1}\\
            \vdots\\
            \theta_{n}
        \end{pmatrix}=\begin{pmatrix}
            2k_{1}\pi\\
            \vdots\\
            2k_{n}\pi
        \end{pmatrix}~\text{or~}
        (A+I_{n})\begin{pmatrix}
            \theta_{1}\\
            \vdots\\
            \theta_{n}
        \end{pmatrix}=\begin{pmatrix}
            2k_{1}\pi\\
            \vdots\\
            2k_{n}\pi
            \end{pmatrix},
    \end{align*}
 where we write $(g^{-1})^{T}:=A=(a_{ij})_{1\le i,j\le n}\ne I_{n}$.    

Note that the map $\mathbb{R}^n\ni(\theta_1,\ldots,\theta_n)\overset{\Phi}{\mapsto} (e^{i\theta_1},\ldots,e^{i\theta_n})\in\mathbb{T}^n$ relates the Lebesgue measure $\lambda$ on $\mathbb{R}^n$ to $\mu$ in the sense that for any Borel set $E\subseteq \mathbb{T}^n$, we have $\mu(E)=\dfrac{1}{(2\pi)^n} \lambda(\Phi^{-1}(E)\cap Q)$, where $Q=[0,2\pi)^n\subset\mathbb{R}^n$. Thus, the Haar measure of the set $\{z\in \mathbb{T}^{n}:gz\sim z\}$ vanishes if the Lebesgue measure of the solution sets to the above two linear systems of equations in $\mathbb{R}^n$ vanishes for all $k_i\in\mathbb{Z}, 1\leq i\leq n$.

For the first system of linear equations, we consider the following three cases. If $1$ is not an eigenvalue of $A$, then the first equation has a unique solution for each given $(k_1,\dots,k_n)$ and hence in this case we have at most countably many points in $\mathbb{T}^n$ with $gz=z$, and thus has measure zero. Now assume $1$ appears as an eigenvalue of $A$, say the multiplicity of $1$ is at most $n-1$, then $(\theta_1,\dots,\theta_n)$ has at most $n-1$  free coordinates for each fixed $(k_1,\dots,k_n)$ and thus the measure of $z$ with $gz=z$
still has measure zero. If the multiplicity of $1$ is $n$, since $A\ne I_n$, then $A$ can not be diagonalizable, and its Jordan canonical form contains at least one Jordan block $J_{m}(1)$ of size $m\ge 2$. Consequently, the coefficient matrix of the linear system satisfies rank $(A-I_n)\ge1$, hence the solution space has dimension at most $n-1$ in $\mathbb{R}^n$, and thus has vanishing Lebesgue measure.

For the second system of linear equations, we can argue similarly as above. If $-1$ is not an eigenvalue of $A$, then the second equation has a unique solution for each fixed $(k_1,\ldots,k_n)$. Thus the measure of $z$ with $gz \sim z$ but $gz\ne z$ is zero. If $A$ has $-1$ as an eigenvalue, note that it must has another eigenvalue which is not $-1$ since $n$ is odd, thus there is at most $n-1$ free coordinates for the solution to the second equation for each fixed $(k_1,\ldots,k_n)$, thus the measure is still zero.
\end{proof}

\begin{prop}\label{prop: version of prop 3.1}
Let $G_{n}=\mathbb{Z}^n\rtimes SL_n(\mathbb{Z}),n\ge3$ and n is odd. Let $P$ be a $G_{n}$-invariant von Neumann subalgebra in $(L(G_{n}),\tau)$. Then 
\begin{itemize}
\item either $P=L(H)$ for some normal subgroup $H\lhd G_{n}$; or,
\item $P=A_d$ for some $d\geq 1$, where $A_d\subset L(d\mathbb{Z}^n)$ is defined by $A_d=\{x\in L(d\mathbb{Z}^n): \tau(xs)=\tau(xs^{-1}), \forall s\in d\mathbb{Z}^n\}$, which agrees with the definition given in Lemma \ref{3.4}.
\end{itemize}
\end{prop}
\begin{proof}
Let $\mathcal{Z}(P)$ be the center of $P$. Since $\mathcal{Z}(P)$ is abelian and hence amenable and $G_n$-invariant, we deduce that $\mathcal{Z}(P) \subseteq L(\text{Rad}(G_n))$ by \cite[Theorem A]{aho}, where $\text{Rad}(G_n)$ denotes the amenable radical of $G_n$. We know that $L(\text{Rad}(G_n)) = L(\mathbb{Z}^n)$ from Lemma \ref{3.3}. Note that as an $SL_n(\mathbb{Z})$-invariant abelian von Neumann subalgebra of $L(\mathbb{Z}^n)$, $\mathcal{Z}(P)\cap L(\mathbb{Z}^n)=L^{\infty}(Y)$ for some (measurable) factor map  $SL_n(\mathbb{Z})\curvearrowright \mathbb{T}^n\rightarrow Y$. 
Note that all (measurable) factors of the standard action $SL_n(\mathbb{Z})\curvearrowright \mathbb{T}^n$ are classified in  \cite[Example 5.9]{wit}, we obtain that $\mathcal{Z}(P)\cap L(\mathbb{Z}^n)=\mathbb{C}$, $L(d\mathbb{Z}^n)$ or $A_d$ for some $d\geq 1$. Then we deduce that $\mathcal{Z}(P)=\mathbb{C}$, $L(d\mathbb{Z}^n)$ or $A_d$ for some $d\geq 1$ from $\mathcal{Z}(P) \subseteq L(\mathbb{Z}^n)$. Since the subfactor case, i.e. the case $\mathcal{Z}(P)=\mathbb{C}$, is already handled by Proposition \ref{prop: classify_invariant_subfactors}. We only need to consider two cases.

\textbf{Case 1.} $\mathcal{Z}(P) = L(d\mathbb{Z}^n)$ for some $d\ge 1$.

\textbf{Claim.} $P = L(d\mathbb{Z}^n)$.

Indeed, note that $    P\subseteq \mathcal{Z}(P)' \cap L(G_{n})$, we just need to check that 
\begin{equation*}
    \mathcal{Z}(P)' \cap L(G_{n}) \subseteq L(\mathbb{Z}^n).
\end{equation*}

Take any $a \in L(d\mathbb{Z}^n)' \cap L(G_{n})$. Write  $a=\sum_{v\in \mathbb{Z}^n, g\in SL_{n}(\mathbb{Z})}\lambda_{v,g}u_{(v,g)}$, $\lambda_{v,g} \in \mathbb{C}$ for its Fourier expansion, where $u_{(v,g)}$ denotes the unitary corresponding to the group element$(v,g)$ and $\sum\limits_{v,g} |\lambda_{v,g}|^2 <\infty$. 

For any $k\in d\mathbb{Z}^n$, the commutation $au_{(k,I_{n})} = u_{(k,I_{n})}a$ gives
\begin{align*}
    u_{(k,I_{n})}au_{(k,I_{n})}^{-1}&=\sum\limits_{v,g}\lambda_{v,g}u_{(k+v-g\cdot k,g)}=a=\sum\limits_{v,g}\lambda_{v,g}u_{(v,g)}\\
    &=\sum\limits_{v,g}\lambda_{(k+v-g\cdot k),g}u_{(k+v-g\cdot k,g)}.
\end{align*}

Comparing coefficients of $u_{(k+v-g\cdot k,g)}$ on both sides, we obtain $\lambda_{v,g}=\lambda_{(k+v-gk),g}$,$\forall k\in d\mathbb{Z}^n$ and $g\in SL_n(\mathbb{Z})$. Now, for fixed $v$ and $g$, consider the set
\begin{align*}
    \#\{k+v-g\cdot k:k\in d\mathbb{Z}^n\}=\begin{cases}
        \ \infty, \ \ \ \ \ g\ne I_{n}\\
        <\infty,\ \ g=I_{n}.
    \end{cases}
\end{align*}
If there exist $g\ne I_{n}$ and some $v\in \mathbb{Z}^n$ such that $\lambda_{v,g} \ne 0$, then by the identity $\lambda_{v,g}=\lambda_{(k+v-gk),g}$,$\forall k\in d\mathbb{Z}^n$, for infinitely many $w\in \{k+v-g\cdot k:k\in d\mathbb{Z}^n\}$, we obtain 
\begin{align*}
    \sum_{w\in \mathbb{Z}^n}|\lambda_{w,g}|^2 \ge \sum\limits_{w\in\{k+v-g\cdot k:k\in d\mathbb{Z}^n\}} |\lambda_{w,g}|^2 = |\lambda_{v,g}|^2\cdot\infty =\infty,
\end{align*}
This contradicts the condition $\sum\limits_{w,g}|\lambda_{w,g}|^2<\infty$. Therefore $\lambda_{v,g} = 0$ for all $g\ne I_{n}$ and $v\in\mathbb{Z}^n$. Consequently, $a = \sum\limits_{v\in \mathbb{Z}^n}\lambda_{v,I_{n}}u_{(v,I_{n})}\in L(\mathbb{Z}^n)$. 

We obtain $P\subseteq \mathcal{Z}(P)' \cap L(G_{n}) = L(\mathbb{Z}^n)$. Hence $P$ is abelian and thus $P =$$~\mathcal{Z}(P) = L(d\mathbb{Z}^n)$.

\textbf{Case 2.} $\mathcal{Z}(P) = A_{d}$ for some $d\ge1$.

\textbf{Claim.} $P=A_d$.

We know that $P\subseteq A_d' \cap L(G_{n}) = L(\mathbb{Z}^n)$ by Lemma \ref{3.4}, which implies $P$ is abelian and hence $P=\mathcal{Z}(P) = A_d$.
\end{proof}

\subsection{The proof for even $n$}

We need to prepare one lemma, which should be compared with Lemma \ref{3.4}.

\begin{lem}\label{3.7}
Let $G_{n}=\mathbb{Z}^n \rtimes SL_n(\mathbb{Z}), n\ge 2$ and $n$ is even. Let $d\ge1$ and $A_{d}:=\{\sum_{v\in d\mathbb{Z}^{n}}c_{v}v:c_{v}=c_{-v}\in\mathbb{C},\forall v\in d\mathbb{Z}^{n}\}\subseteq L(d\mathbb{Z}^n)$. Then $A_{d}' \cap L(G_{n}) \subset L(\mathbb{Z}^n \rtimes \{\pm I_{n}\})$.
\end{lem}
\begin{proof}
Write $X = \mathbb{T}^n$ and $Y = \widehat{d\mathbb{Z}^{n}}/{\sim}$. Without loss of generality, we may assume that $d=1$. Note that $\sim$ is defined on $\mathbb{T}^n$ via $z\sim w$, where $z=(z_1,\ldots,z_n), w=(w_1,\ldots ,w_n)\in \mathbb{T}^n$, iff $z_i=w_i$ for all $1\le i \le n$ or $z_iw_i=1$  for all $1\le i \le n$. Hence we have $L(G_n) \cong L^{\infty}(X, \mu) \rtimes SL_n(\mathbb{Z})$ and $A_d \cong L^{\infty}(Y, \nu) $, where $\mu$ is the Haar measure on $X$ and $\nu$ is the pushforward measure $\pi_*\mu$ on $Y$ with respect to the factor map  $\pi : SL_n(\mathbb{Z})\curvearrowright (X,\mu) \rightarrow (Y,\nu)$. 

Let $E:L^{\infty}(X, \mu) \rightarrow L^{\infty}(Y, \nu)$ be a conditional expectation. Then $E(f)(y) = \int_X f(x)d\mu_{y}(x)$, where $f\in L^{\infty}(X, \mu)$ and $\mu=\int_Y \mu_{y}d\nu(y)$ is the measure decompositionwith respect to $\pi$.
Note that $E$ satisfies $E|_{L^{\infty}(Y,\nu)} = id$ and $E(\xi f) = E(\xi)f,~\forall \xi \in L^{\infty}(X, \mu),~f \in L^{\infty}(Y, \nu)$. Note that $E$ is also faithful, i.e. if $f\ge0$ in $ L^{\infty}(X, \mu)$ and $E(f)=0$, then $f=0$. 
Indeed, $E(f) = 0 \Leftrightarrow \int_X f(x)d\mu_y(x)= 0$ for $\nu$-a.e. $y$. Since $f(x)\ge 0$ for $\mu$-a.e, we deduce that $f(x) = 0$ for $\mu_{y}$-a.e. $x$ and $\nu$-a.e. $y$. Then $\mu\{x:f(x)\ne 0\} = \int_Y \mu_{y}(\{x:f(x)\ne 0\})d\nu(y) = 0$, i.e. $f=0$. 

Take any $a \in L^{\infty}(Y, \nu)' \cap [L^{\infty}(X,\mu)\rtimes SL_{n}(\mathbb{Z})]$. Write  $a=\sum_{g\in SL_{n}(\mathbb{Z})}f_{g}g$ for its Fourier expansion, where $f_g \in L^{\infty}(X, \mu)$ and $g$ is an abbreviation for the unitary $u_g = \sigma_{g}\otimes \lambda_g \in B(L^2(X,\mu)\otimes l^2(SL_n(\mathbb{Z})))$. For any $\xi\in  L^{\infty}(Y, \nu)$, the commutation $a\xi = \xi a$ gives
\begin{align*}
    a\xi&=(\sum\limits_{g}f_g g)\xi=\sum\limits_{g}(f_g \sigma_{g}(\xi))g,\\
    \xi a&=\sum\limits_{g}( \xi f_g )g.
\end{align*}

Comparing coefficients of $g$, we obtain that $f_g(\sigma_g(\xi)-\xi)=0$, hence $f_g^*f_g(\sigma_g(\xi)-\xi)=0$, where $f_g^*$ denotes the complex conjugate of $f_g$. Apply $E(\cdot)$ on both sides, we get that 
\begin{equation}\label{lemma7 equation 1}
0=E(f_g^*f_g(\sigma_g(\xi)-\xi)) = E(f_g^*f_g)(\sigma_g(\xi)-\xi),
\end{equation} for all $g\in SL_{n}(\mathbb{Z}) $ and $\xi \in L^{\infty}(Y,\nu)$.
The last equality holds because $\sigma_g(\xi)-\xi \in L^{\infty}(Y, \nu)$.

The following observation is an analogue of a piece argument used in Lemma \ref{3.4} for even $n$ with a similar proof. Below, we write $s=I_n$, the $n\times n$ identity matrix in $SL_n(\mathbb{Z})$.

\begin{observation}\label{observation2}
    For any $g\in SL_{n}(\mathbb{Z})\setminus \{I_n,s\}$, then $ \mu(\{z\in \mathbb{T}^{n}:gz\sim z\})=0$.
\end{observation}
\begin{proof} Write $z=(z_1,\dots,z_n)\in \mathbb{T}^n$, where $z_{k} = e^{i\theta_{k}}$ and $\theta_k \in \mathbb{R}$ for $1\le k \le n$. Then 
note that $gz\sim z$ iff we have the following equation holds, where $(g^{-1})^{T}:=A=(a_{ij})_{1\le i,j\le n}\ne I_{n}$ or $s$, 
 \begin{align*}
        (A-I_{n})\begin{pmatrix}
            \theta_{1}\\
            \vdots\\
            \theta_{n}
        \end{pmatrix}=\begin{pmatrix}
            2k_{1}\pi\\
            \vdots\\
            2k_{n}\pi
        \end{pmatrix}~\text{or~}
        (A+I_{n})\begin{pmatrix}
            \theta_{1}\\
            \vdots\\
            \theta_{n}
        \end{pmatrix}=\begin{pmatrix}
            2k_{1}\pi\\
            \vdots\\
            2k_{n}\pi
            \end{pmatrix},
    \end{align*}\\
for some $k_i\in \mathbb{Z}, 1\le i \le n$.    

For the first system of linear equations, following exactly the same proof as in Lemma \ref{3.4}, we consider the case where the eigenvalue $1$ appears and conclude that  $ \mu(\{z\in \mathbb{T}^{n}:gz = z\})=0$.

For the second system of equations, we consider the following three cases. If $-1$ is not an eigenvalue of $A$, then the first equation has a unique solution for each given $(k_1,\dots,k_n)$ and hence in this case we have at most countably many points in $\mathbb{T}^n$ with $gz\sim z$ but $gz\ne z$, and thus has measure zero. Now assume $-1$ appears as an eigenvalue of $A$, say the multiplicity of $-1$ is at most $n-1$, then $(\theta_1,\dots,\theta_n)$ has at most $n-1$  free coordinates for each fixed $(k_1,\dots,k_n)$ and thus the set of $z$ with $gz\sim z$ but $gz\ne z$ still has measure zero. If the multiplicity of $-1$ is $n$, then since $A\ne s$, then $A$ can not be diagonalizable, and its Jordan canonical form contains at least one Jordan block $J_{m}(1)$ of size $m\ge 2$. Consequently, the coefficient matrix of the linear system satisfies rank $(A+I_n)\ge1$, hence the solution space has dimension at most $n-1$, thus again the measure  of $z$ satisfying $gz\sim z$ but $gz\ne z$ is zero.
\end{proof}
Combining  \eqref{lemma7 equation 1} with the above observation, we claim that $E(f^*_g f_g) = 0$ for all $g\not\in\{I_n, s\}$. Since $E$ is faithful, we obtain $f_g = 0$ for all $g\not\in\{I_n, s \}$, which shows that $a\in L^\infty(X) \rtimes \{\pm I_n\} \cong L(\mathbb{Z}^n \rtimes \{\pm I_n\})$.

Indeed, assume that for some $g\ne I_n$ and $s$, we have $E(f^*_g f_g) \ne 0 $, this implies $\nu(U)>0$, where $U:=\{y\in Y~|~E(f^*_g f_g)(y) \ne 0\}$. By Observation \ref{observation2}, $\nu\{y\in Y~|~gy=y\} = 0$. Reference \cite[Proposition 4.22]{dl} tells us that there exists some $B\subseteq Y$ such that $B\subseteq U$ with $\nu(B)\ge 0$ and $gB\cap B = \phi $. Set $\xi = \chi_{B}$, the characteristic funtion on $B$, then for a.e. $y\in B$, we get $0 = E(f^*_g f_g) (y)(\sigma_g(\xi)-\xi)(y)=-E(f^*_g f_g) (y)$, a contradiction.
\end{proof}

The following lemma will be applied to $Q=\mathcal{Z}(P)$, where $P$ is a $G_n$-invariant von Neumann subalgebra in $L(G_n)$ for even $n$.

\begin{lem}\label{3.5}Let $n\geq 2$ be an even number.
Let $Q\subsetneq L(\mathbb{Z}^{n} \rtimes \{\pm I_{n}\})$ be a $G_{n}$-invariant abelian von Neumann subalgebra. Then $Q\subseteq L(\mathbb{Z}^n)$ and hence $Q=\mathbb{C}$, $L(d\mathbb{Z}^n)$ or $A_d$ for some $d\geq 1$.    
\end{lem}
%%%%%%%%%%%%%%%%%%%%%%%%%%%%%%%%%%%%%%%%%%%%
\begin{proof}
Let $E: L(\mathbb{Z}^n\rtimes \{\pm{I_{n}}\})\rightarrow Q$ be the trace $\tau$-preserving conditional expectation onto $Q$. Below, we directly write $g$ for the canonical unitary $u_g$ inside $L(G_{n})$ for simplicity and we use the notation $\langle a, b \rangle$ to mean $\tau(b^*a)$  for any $a, b\in L(G_{n})$. And we denote $s=-I_{n}\in SL_{n}(\mathbb{Z})$.

\textbf{Step 1.} We show that $E(s)=0$ using the abelian assumption of $Q$.

Note that $L(SL_n(\mathbb{Z}))'\cap L(\mathbb{Z}^n\rtimes SL_n(\mathbb{Z}))=L(\{\pm I_{n}\})$, say by a direct calculation using \cite[Lemma 2.7]{aj}.

Since $Q$ is $G_{n}$-invariant, we deduce that $gE(s)g^{-1}=E(gsg^{-1})=E(s)$ for all $g\in SL_n(\mathbb{Z})$. Hence, 
$E(s)\in L(SL_n(\mathbb{Z}))'\cap L(\mathbb{Z}^n\rtimes SL_n(\mathbb{Z}))=L(\{\pm I_{n}\})=L(\langle s\rangle)$. Therefore, we may write $E(s)=\lambda+\mu s$, where $\lambda,\mu\in\mathbb{C}$. Observe that $\lambda=\tau(E(s))=\tau(s)=0$. Hence, $\mu^2=(\mu s)^2=E(s)E(s)=E(sE(s))=E(s\mu s)=\mu$. It follows that $\mu=0$ or $1$. Hence, $E(s)=0$ or $s$. But $E(s)=s$ is impossible for the following reason. Assume that $E(s)=s$, then we have $s\in Q$. And $Q$ is $G_{n}$-invariant, which implies that $(as)E(s)(as)^{-1} = (as)s(as)^{-1} =(2a)s\in Q$ for all $a\in \mathbb{Z}^{n}$.
Then we can deduce that $2a = [(2a)s]s=s[(2a)s]=(-2a)ss=-2a$ for all $a\in \mathbb{Z}^{n}$, where the second equality holds since $Q$ is abelian. This gives us a contradiction. Hence we have shown that $E(s)=0$, finishing Step 1.

\textbf{Step 2.} We show that $Q\cap L(\mathbb{Z}^n)=Q$.

Since $Q\cap L(\mathbb{Z}^n)$ is an abelian von Neumann subalgebra in $L(\mathbb{Z}^n)$ which is $SL_n(\mathbb{Z})$-invariant, it equals $L^{\infty}(Y)$ for some (measurable) factor map $SL_n(\mathbb{Z})\curvearrowright \mathbb{T}^n\rightarrow Y$. According to the full classification of all such (measurable) factor maps as in
 \cite[Example 5.9]{wit}, we can deduce that $Q\cap L(\mathbb{Z}^n)=\mathbb{C}$, $L(d\mathbb{Z}^n)$ or $A_d$ for some $d\geq 1$. Thus we split the proof by considering these three possibilities and argue that in fact $Q\cap L(\mathbb{Z}^n)=Q$ always holds ture.

\textbf{Case 1.} $Q\cap L(\mathbb{Z}^n)=\mathbb{C}$.

We claim that $Q=\mathbb{C}$.

First, observe that for any $v\in \mathbb{Z}^n$, we have 
\begin{align*}
E(v)\in Q\cap L(\mathbb{Z}^n)'=Q\cap L(\mathbb{Z}^n)=\mathbb{C},
\end{align*}
where to get the 2nd equality, we have used the fact that $L(\mathbb{Z}^n)$ is a MASA (maximal abelian von Neumann subalgebra) in $L(G_{n})$.
Thus, $\forall v\in \mathbb{Z}^n\setminus \{0\}$, we get $E(v)=\tau(E(v))=\tau(v)=0$.

We are left to show $E(vs)=0$ for all $v\in\mathbb{Z}^n\setminus \{0\}$. By Lemma \ref{3.6}, it suffices to show that $E(e_1s)=0$. 

First, notice that
\begin{align*}
E(e_1s)\in L(\{\left(\begin{smallmatrix}
1&\boldsymbol{z}\\
0&I_{n-1}
\end{smallmatrix}\right) \})'\cap L(\mathbb{Z}^n\rtimes \{\pm I_{n}\})\subseteq L((\mathbb{Z},0,\dots, 0)^t)\rtimes \{\pm I_{n}\}.
\end{align*}
where $\boldsymbol{z}\in M_{1, n-1}(\mathbb{Z})$. Thus, we may write $E(e_1s)=a+bs$, where $a, b\in L((\mathbb{Z}, 0,\dots,0)^t)$.

From $\langle v-E(v), E(e_1s)\rangle=0$, we get that $\langle v, a \rangle=0$ for all $v\in\mathbb{Z}^n\setminus \{0\}$. Hence, $a\in\mathbb{C}$. Then by computing the trace of $E(e_1s)$, we get that $a=\tau(a+bs)=\tau(E(e_1s))=\tau(e_1s)=0$. Hence, $E(e_1s)=bs$. 

Let us write $b=\sum_{n\in\mathbb{Z}}\mu_ne_n$, where $\mu_n\in\mathbb{C}$ and $e_n=(n,0,\ldots, 0)^t$. Set $f_n=(0, n,0,\dots,0)^t\in\mathbb{Z}^n$. 
Notice that $f_n=g\cdot e_n$, where $g = \left(\begin{smallmatrix}
 0 & -1 \\
1 & 0 \\
     & & I_{n-2}
\end{smallmatrix}\right)\in SL_{n}(\mathbb{Z})$. Thus,
$E(f_1s)=E((g\cdot e_{1})s)=E(ge_{1}g^{-1}s)=gE(e_{1}s)g^{-1}=\sum_{n\in\mathbb{Z}}\mu_n f_ns$. 

By $Q$-bimodule property of $E$, we have $E(e_1s)E(f_1s)=E(e_1sE(f_1s))$. Let us compute both sides concretely.
\begin{align*}
E(e_1s)E(f_1s)&=\sum_{n, m\in\mathbb{Z}}\mu_m\mu_n(m,-n,0,\ldots, 0)^t ,\\
E(e_1sE(f_1s))&=E(e_1s(\sum_{m\in\mathbb{Z}}\mu_mf_ms))=\sum_{m\in\mathbb{Z}}\mu_mE((1,-m,0,\ldots,0)^t)=0,
\end{align*}
where to get the last equality, we used the fact that $E(v)=0$ for all $v\in\mathbb{Z}^n\setminus\{0\}$.
Hence, $\mu_m\mu_n=0$ for all $(m, n)\in\mathbb{Z}^2$. Thus, $\mu_n^2=0$, i.e. $\mu_n=0$ for all $n\in\mathbb{Z}$; equivalently, $b=0$ and thus $E(e_1s)=0$. The proof of this case is done.

\textbf{Case 2.} $Q\cap L(\mathbb{Z}^n)=L(d\mathbb{Z}^n)$ for some $d\geq 1$.

We claim that $Q=L(d\mathbb{Z}^n)$. 

If $d=1$, then $L(\mathbb{Z}^n)\subseteq Q$ and thus $E(v)=v$ for all $v\in\mathbb{Z}^n$. Hence  $E(vs)=vE(s)=0$. Thus, $Q=L(\mathbb{Z}^n)$.

From now on, we assume that $d\geq 2$.
The proof given below is essentially the same as the proof of Case 1 with minor modification, we record it for completeness.

First, for any $v\in\mathbb{Z}^n\setminus {d\mathbb{Z}^n}$, we observe that $E(v)=0$.

Indeed, 
\begin{align*}
E(v)\in Q\cap L(\mathbb{Z}^n)'=Q\cap L(\mathbb{Z}^n)=L(d\mathbb{Z}^n).
\end{align*}
Thus, from $\langle v-E(v), E(v)\rangle=0$, we deduce that $\langle E(v), E(v) \rangle=0$, i.e. $E(v)=0$.

Thus, for any $v\in \mathbb{Z}^n$, we have either $E(v)=0$ or $v\in d\mathbb{Z}^n$ and in this case $E(v)=v$.

We are left to show $E(vs)=0$ for all $v\in\mathbb{Z}^n$. It suffices to show $E(e_1s)=0$ by Lemma \ref{3.6}.

First, notice that
\begin{align*}
E(e_1s)\in L(\{\left(\begin{smallmatrix}
1&\boldsymbol{z}\\
0&I_{n-1}
\end{smallmatrix}\right) \})'\cap L(\mathbb{Z}^n\rtimes \{\pm I_{n}\})\subseteq L((\mathbb{Z},0,\dots, 0)^t)\rtimes \{\pm I_{n}\}.
\end{align*}
where $\boldsymbol{z}:=(\mathbb{Z},\mathbb{Z},\dots,\mathbb{Z})_{1\times n-1}$.
Thus, we may write $E(e_1s)=a+bs$, where $a, b\in L((\mathbb{Z},0,\dots, 0)^t)$.

From $\langle v-E(v), E(e_1s)\rangle=0$, we deduce that $\langle v, a \rangle=0$ for all $v\in \mathbb{Z}^n\setminus {d\mathbb{Z}^n}$. Hence, $a\in L((d\mathbb{Z},0,\dots,0)^t)\subset L(d\mathbb{Z}^n)\subseteq Q$. Similarly, from $\langle e_1s-E(e_1s), a \rangle=0$, we deduce that $\langle a, a \rangle=0$, i.e. $a=0$.
Hence, $E(e_1s)=bs$.

Then by repeating the last part of the proof of Case 1, we deduce that $b=0$.

\textbf{Case 3.} $Q\cap L(\mathbb{Z}^n)=A_d$ for some $d\geq 1$.

We claim that $Q=A_d$. 

First, fix any nonzero vector $v\in d\mathbb{Z}^n$, we have $E(v)\in Q\cap L(\mathbb{Z}^n)'=Q\cap L(\mathbb{Z}^n)=A_d$. Hence, we can write $E(v)=\sum\limits_{\omega\in d\mathbb{Z}^n}\lambda_{\omega}\omega$ with symmetric Fourier coefficients, i.e. $\lambda_{\omega}=\lambda_{\omega^{-1}}$ for all $\omega\in d\mathbb{Z}^n$. Note that here we use $\omega^{-1}$ to mean the inverse of $\omega$ in $d\mathbb{Z}^n$.

From $0=\langle v-E(v), Q \rangle$, we deduce that $\langle v-\sum\limits_{\omega\in d\mathbb{Z}^n}\lambda_{\omega}\omega, \omega_0+\omega_{0}^{-1} \rangle=0$ for all $\omega_0\in d\mathbb{Z}^n$. Observe that by taking $\omega_0\not\in\{v, v^{-1}\}$, we can get that $0=\lambda_{\omega_0}+\lambda_{\omega_0^{-1}}=2\lambda_{\omega_0}.$ Similarly, by taking $\omega_0=v$, we deduce that $\lambda_{v}=\lambda_{v^{-1}}=\frac{1}{2}$. Hence $E(v)=\frac{v+v^{-1}}{2}$ for all $v\in d\mathbb{Z}^n$.

Second, we check that $E(v)=0$ for all $v\in\mathbb{Z}^n\setminus {d\mathbb{Z}^n}$.

Observe that we still have $E(v)\in Q\cap L(\mathbb{Z}^n)'=Q\cap L(\mathbb{Z}^n)=A_d$. Hence, from the fact that $0=\langle v-E(v), E(v) \rangle$, we deduce that $\langle E(v), E(v)\rangle=0$ for all $v\in \mathbb{Z}^n\setminus {d\mathbb{Z}^n}$ since $\langle v, E(v) \rangle=0$ for such a $v$, thus $E(v)=0$ is proved.

We are left to show $E(e_1s)=0$ by Lemma \ref{3.6}. 

Once again, we still have that for all $z\in M_{1, n-1}(\mathbb{Z})$,
\begin{align*}
E(e_1s)\in L(\{\left(\begin{smallmatrix}
1&\boldsymbol{z}\\
0&I_{n-1}
\end{smallmatrix}\right) \})'\cap L(\mathbb{Z}^n\rtimes \{\pm I_{n}\})\subseteq L((\mathbb{Z},0,\dots, 0)^t)\rtimes \{\pm I_{n}\}.
\end{align*}
Thus, we may write $E(e_1s)=a+bs$, where $a, b\in L((\mathbb{Z},0,\dots, 0)^t)$.

From $\langle v-E(v), E(e_1s)\rangle=0$, we deduce that $\langle v, a \rangle=0$ for all $v\in \mathbb{Z}^n\setminus {d\mathbb{Z}^n}$. Hence, $a\in L((d\mathbb{Z},0,\dots,0)^t)$.

Next, from $\langle e_1s-E(e_1s), A_d \rangle=0$, we deduce that $\langle a, A_d \rangle=0$, equivalently, $a+\sigma_s(a)=0$. In other words, if we write $a=\sum_{i\in\mathbb{Z}}\lambda_ie_{di}$, where $e_{di}=(di,0,\ldots, 0)^t$, then 
\begin{align}\label{eq: skew-symmetry on the coefficients of a}
\lambda_i+\lambda_{-i}=0,\forall~i\in\mathbb{Z}. 
\end{align}

Now, let us write $a=\sum_{i\in\mathbb{Z}}\lambda_ie_{di}$ and $b=\sum_{j\in\mathbb{Z}}\mu_je_j$, where $\lambda_i,\mu_j\in \mathbb{C}$ for all $i, j$.

Thus, $E(f_1s)=E(ge_{1}g^{-1}s) = g(a+bs)g^{-1}=(\sum_{i\in\mathbb{Z}}\lambda_if_{di})+(\sum_{j\in\mathbb{Z}}\mu_jf_j)s$, where $g = \left(\begin{smallmatrix}
 0 & -1 \\
1 & 0 \\
     & & I_{n-2}
\end{smallmatrix}\right)\in SL_{n}(\mathbb{Z})$ and $f_j=(0,j,0,\ldots, 0)^t\in\mathbb{Z}^n$.

Next, note that
\begin{align*}
E((e_1+e_{-1})s)&=E(e_1s)+E(e_{-1}s)=E(e_1s)+\sigma_s(E(e_1s))\\
&=(a+bs)+\sigma_s(a+bs)
=(a+\sigma_s(a))+(b+\sigma_s(b))s=(b+\sigma_s(b))s.
\end{align*}
Meanwhile, since $b+\sigma_s(b)\in A_d\subseteq Q$, we deduce that $(b+\sigma_s(b))s=E((e_1+e_{-1})s)=E(E((e_1+e_{-1})s))=E((b+\sigma_s(b))s)=(b+\sigma_s(b))E(s)=0$, i.e. $b+\sigma_s(b)=0$, equivalently, 
\begin{align}\label{eq: skew-symmetry on coefficients of b}
\mu_k+\mu_{-k}=0, \forall ~k\in \mathbb{Z}.
\end{align}

Next, we compute both sides of the identity $E(e_1s)E(f_1s)=E(e_1sE(f_1s))$ concretely.

On the one hand, since
$E(e_1s)=a+bs=\sum_{i\in\mathbb{Z}}\lambda_ie_{di}+\sum_{j\in\mathbb{Z}}\mu_je_js$, we get that
\begin{align*}
E(e_1s)E(f_1s)&=\left[\sum_{i, j}\lambda_i\lambda_j\left(\begin{smallmatrix}
di\\
dj\\
0\\
\vdots\\
0
\end{smallmatrix}\right)+\sum_{j, k}\mu_j\mu_k\left(\begin{smallmatrix}
j\\
-k\\
0\\
\vdots\\
0
\end{smallmatrix}\right)\right]\\
&+\left[\sum_{i,j}\lambda_i\mu_j\left(\begin{smallmatrix}
di\\
j\\
0\\
\vdots\\
0
\end{smallmatrix}\right)+\sum_{i, j}\lambda_i\mu_j\left(\begin{smallmatrix}
j\\
-di\\
0\\
\vdots\\
0
\end{smallmatrix}\right)\right]s;
\end{align*}
on the other hand, we have
\begin{align*}
E(e_1sE(f_1s))&=E(e_{1}s\left[(\sum_{i\in\mathbb{Z}}\lambda_if_{di})+(\sum_{j\in\mathbb{Z}}\mu_jf_j)s\right])\\
&=\sum_i\lambda_iE(\left(\begin{smallmatrix}
1\\
-di\\
0\\
\vdots\\
0
\end{smallmatrix}\right)s)+\sum_j\mu_jE(\left(\begin{smallmatrix}
1\\
-j\\
0\\
\vdots\\
0
\end{smallmatrix}\right))\\
&=\sum_i\lambda_i g_{i}E(e_1s)g_{i}^{-1}
+\sum_j\mu_jE(\left(\begin{smallmatrix}
1\\
-j\\
0\\
\vdots\\
0
\end{smallmatrix}\right))\\
&=\sum\limits_{i} \lambda_{i} \left[\sum\limits_{j} \lambda_{j}g_{i}e_{dj}g_{i}^{-1} + \sum\limits_{k} \mu_{k}g_{i}e_{k}g_{i}^{-1}s\right] + \sum_j\mu_jE(\left(\begin{smallmatrix}
1\\
-j\\
0\\
\vdots\\
0
\end{smallmatrix}\right))\\
&=\sum_{i,j}\lambda_i\lambda_j\left(\begin{smallmatrix}
dj\\
-d^2ij\\
0\\
\vdots\\
0
\end{smallmatrix}\right)+\sum_{i, k}\lambda_i\mu_k\left(\begin{smallmatrix}
k\\
-dik\\
0\\
\vdots\\
0
\end{smallmatrix}\right)s+\sum_j\mu_jE(\left(\begin{smallmatrix}
1\\
-j\\
0\\
\vdots\\
0
\end{smallmatrix}\right)),
\end{align*}
where $g_{i} = \left(\begin{smallmatrix}
 1 & 0 \\
-di & 1 \\
     & & I_{n-2}
\end{smallmatrix}\right) \in SL_{n}(\mathbb{Z})$.

Due to the difference in computing $E((1,-j,0,\ldots, 0)^t))$ when $d=1$ or $d\ge2$, we need to split the proof by considering two subcases.

\textbf{Subcase 3-I}: $d=1$.

In this case, since we have proved that $E(v)=\frac{v+v^{-1}}{2}$ for all $v\in d\mathbb{Z}^n=\mathbb{Z}^n$, we get that $$E((1,-j,0,\ldots,0)^t)=\frac{1}{2}\left[(1,-j,0,\ldots,0)^t+(-1,j,0,\ldots,0)^t)\right].$$ 
Then we may continue the above calculation to deduce that 
\begin{align*}
E(e_1sE(f_1s))&=\sum_{i,j}\lambda_i\lambda_j\left(\begin{smallmatrix}
j\\
-ij\\
0\\
\vdots\\
0
\end{smallmatrix}\right)+\sum_{i, k}\lambda_i\mu_k\left(\begin{smallmatrix}
k\\
-ik\\
0\\
\vdots\\
0
\end{smallmatrix}\right)s+\sum_j\mu_j\frac{1}{2}\left[\left(\begin{smallmatrix}
1\\
-j\\
0\\
\vdots\\
0
\end{smallmatrix}\right)+\left(\begin{smallmatrix}
-1\\
j\\
0\\
\vdots\\
0
\end{smallmatrix}\right)\right],\\
E(e_1s)E(f_1s)&=\left[\sum_{i, j}\lambda_i\lambda_j\left(\begin{smallmatrix}
i\\
j\\
0\\
\vdots\\
0
\end{smallmatrix}\right)+\sum_{j, k}\mu_j\mu_k\left(\begin{smallmatrix}
j\\
-k\\
0\\
\vdots\\
0
\end{smallmatrix}\right)\right]\\
&+\left[\sum_{i,j}\lambda_i\mu_j\left(\begin{smallmatrix}
i\\
j\\
0\\
\vdots\\
0
\end{smallmatrix}\right)+\sum_{i, j}\lambda_i\mu_j\left(\begin{smallmatrix}
j\\
-i\\
0\\
\vdots\\
0
\end{smallmatrix}\right)\right]s.
\end{align*}
Therefore, we can deduce that
\begin{align}
\label{eq-a: d=1 case}\sum_{i,j}\lambda_i\lambda_j\left(\begin{smallmatrix}
j\\
-ij\\
0\\
\vdots\\
0
\end{smallmatrix}\right)+\sum_j\mu_j\frac{1}{2}\left[\left(\begin{smallmatrix}
1\\
-j\\
0\\
\vdots\\
0
\end{smallmatrix}\right)+\left(\begin{smallmatrix}
-1\\
j\\
0\\
\vdots\\
0
\end{smallmatrix}\right)\right]=\sum_{i, j}\lambda_i\lambda_j\left(\begin{smallmatrix}
i\\
j\\
0\\
\vdots\\
0
\end{smallmatrix}\right)+\sum_{j, k}\mu_j\mu_k\left(\begin{smallmatrix}
j\\
-k\\
0\\
\vdots\\
0
\end{smallmatrix}\right),\\
\label{eq-b: d=1 case}\sum_{i, k}\lambda_i\mu_k\left(\begin{smallmatrix}
k\\
-ik\\
0\\
\vdots\\
0
\end{smallmatrix}\right)=\sum_{i,j}\lambda_i\mu_j\left(\begin{smallmatrix}
i\\
j\\
0\\
\vdots\\
0
\end{smallmatrix}\right)+\sum_{i, j}\lambda_i\mu_j\left(\begin{smallmatrix}
j\\
-i\\
0\\
\vdots\\
0
\end{smallmatrix}\right).
\end{align}

Now we can compare the coefficients of $(1,j,0,\ldots,0)^t$ and $(-1,j,0,\ldots,0)^t$ respectively on both sides of  \eqref{eq-a: d=1 case} to deduce that
\begin{align*}
\lambda_1\lambda_{-j}+\frac{1}{2}\mu_{-j}&=\lambda_1\lambda_j+\mu_1\mu_{-j},\\
\lambda_j\lambda_{-1}+\frac{1}{2}\mu_j&=\lambda_{-1}\lambda_j+\mu_{-1}\mu_{-j}.
\end{align*}
Taking the sum of the above two equations and applying \eqref{eq: skew-symmetry on the coefficients of a} and \eqref{eq: skew-symmetry on coefficients of b}, we get that $2\lambda_1\lambda_{-j}=0$ for all $j\in\mathbb{Z}$. Hence $\lambda_1=0$. Then plugging it in the above first equation to get that $\frac{1}{2}\mu_{-j}=\mu_1\mu_{-j}$. Thus either $\mu_1=\frac{1}{2}$ or $\mu_{-j}=0$ for all $j\in \mathbb{Z}$. 

Once we have $\mu_j=0$ for all $j\in\mathbb{Z}$, i.e. $b=0$, then $a=E(e_1s)=E(E(e_1s))=E(a)\in Q\cap L(\mathbb{Z}^n)'=Q\cap L(\mathbb{Z}^n)=A_1$. Recall that \eqref{eq: skew-symmetry on the coefficients of a} holds, i.e. $\langle a, A_1 \rangle=0$. Hence $a=0$; equivalently, $E(e_1s)=0$.
Therefore, we may without loss of generality assume that $\mu_1=\frac{1}{2}$ and try to deduce a contradiction.

We compute the coefficient of $(-1,j,0,\ldots,0)^t$ on both sides of \eqref{eq-b: d=1 case} to get that
$\lambda_j\mu_{-1}=\lambda_{-1}\mu_j+\lambda_{-j}\mu_{-1}$. By plugging $\lambda_{-1}=-\lambda_1=0$, $\mu_{-1}=-\mu_1=-\frac{1}{2}$ in it and applying \eqref{eq: skew-symmetry on the coefficients of a}, we get $\lambda_j=0$ for all $j\in \mathbb{Z}$, i.e. $a=0$. 

Then for any $|i|\neq 0, 1$, we compute the coefficients of $(i,j,0,\ldots, 0)^t$ on both sides of
\eqref{eq-a: d=1 case} and use $\lambda_j=0$ for all $j\in\mathbb{Z}$ to get that
$0=\lambda_i\lambda_j+\mu_i\mu_{-j}=\mu_i\mu_{-j}$ for all $j\in\mathbb{Z}$. Therefore, $\mu_i=0$ for all $|i|\neq 0, 1$. Recall that $\mu_0+\mu_{-0}=0$, i.e. $\mu_0=0$, hence, we have shown that
$E(e_1s)=0+\frac{1}{2}(e_{1}-e_{-1})s$.

Then, $E(e_{3}s)=e_{1}E(e_1s)e_{-1}=\frac{1}{2}(e_{3}-e_{1})s$.
Hence, 
\begin{align*}
Q\ni E(e_{3}s)E(e_1s)&=\frac{1}{4}(e_{3}-e_{1})s(e_{1}-e_{-1})s\\
&=\frac{1}{4}(e_{3}-e_{1})(e_{-1}-e_{1})\\
&=\frac{1}{4}(e_{2}-e_{0}-e_{4}+e_{2})\in L(\mathbb{Z}^n)\setminus A_1.
\end{align*}
This contradicts to the assumption that $Q\cap L(\mathbb{Z}^n)=A_1$.

\textbf{Subcase 3-II}: $d\geq 2$.

In this case, $(1,-j,0,\ldots,0)^t\not\in d\mathbb{Z}^n$ and hence $ E((1,-j,0,\ldots,0)^t)=0$ since we have proved before that $E(v)=0$ for all $v\in \mathbb{Z}^n\setminus {d\mathbb{Z}^n}$. Thus we get the following identities by comparing the computation of $E(e_1s)E(f_1s)$ and $E(e_1sE(f_1s))$:
\begin{align}
\label{eq1: two sides, n>1} \sum_{i, j}\lambda_i\lambda_j\left(\begin{smallmatrix}
di\\
dj\\
0\\
\vdots\\
0
\end{smallmatrix}\right)+\sum_{j, k}\mu_j\mu_k\left(\begin{smallmatrix}
j\\
-k\\
0\\
\vdots\\
0
\end{smallmatrix}\right)=\sum_{i,j}\lambda_i\lambda_j\left(\begin{smallmatrix}
dj\\
-d^2ij\\
0\\
\vdots\\
0
\end{smallmatrix}\right),\\
\label{eq2: two sides, n>1}\sum_{i,j}\lambda_i\mu_j\left(\begin{smallmatrix}
di\\
j\\
0\\
\vdots\\
0
\end{smallmatrix}\right)+\sum_{i, j}\lambda_i\mu_j\left(\begin{smallmatrix}
j\\
-di\\
0\\
\vdots\\
0
\end{smallmatrix}\right)=\sum_{i, k}\lambda_i\mu_k\left(\begin{smallmatrix}
k\\
-dik\\
0\\
\vdots\\
0
\end{smallmatrix}\right).
\end{align}

By \eqref{eq1: two sides, n>1}, we deduce that $\mu_j\mu_k=0$ for all $(j,-k)^t\not\in d\mathbb{Z}^2$. In particular, $\mu_j=0$ for all $j\not\in d\mathbb{Z}$, hence, $b\in L((d\mathbb{Z}, 0,\dots,0)^t)$.

For any $i\neq 0$ and $j\in\mathbb{Z}$, by comparing the coefficients of $(di,dj,0,\ldots,0)^t$ on both sides of the two identities \eqref{eq1: two sides, n>1} and \eqref{eq2: two sides, n>1}, we deduce that 
\begin{align}
\label{eq11: simplied eq1} \lambda_i\lambda_j+\mu_{di}\mu_{-dj}&=\lambda_{-\frac{j}{di}}\lambda_i, \forall~ i\neq 0, \forall ~j\\
\label{eq22: simplied eq2} \lambda_i\mu_{dj}+\lambda_{-j}\mu_{di}&=\lambda_{-\frac{j}{di}}\mu_{di}, \forall~i\neq 0,\forall~j.
\end{align}
Here, $\lambda_{-\frac{j}{di}}$ is understood as 0 if $(di)\nmid j$.

Substitute $j=1$ into \eqref{eq11: simplied eq1} and 
\eqref{eq22: simplied eq2} and use \eqref{eq: skew-symmetry on coefficients of b} to deduce that 
\begin{align}
\label{eq3: j=1 case} \lambda_i\lambda_1=\mu_{di}\mu_{d}, \forall~ i\neq 0,\\
\label{eq4: j=1 case} \lambda_i\mu_d=\lambda_1\mu_{di},\forall~i\neq 0.
\end{align}
This implies that $(\lambda_i^2-\mu_{di}^2)\lambda_1\mu_d=0$ and $\lambda_1^2=\mu_d^2$ (by plugging $i=1$ in \eqref{eq3: j=1 case}). 

By plugging $j=dik$ in \eqref{eq11: simplied eq1} and \eqref{eq22: simplied eq2} and using \eqref{eq: skew-symmetry on the coefficients of a}, we get that
\begin{align}
\label{eq5} \lambda_i\lambda_{dik}-\mu_{di}\mu_{d^2ik}=\lambda_{-k}\lambda_i, \forall ~i\neq 0, \forall~ k, \\
\label{eq6} \lambda_i\mu_{d^2ik}-\lambda_{dik}\mu_{di}=\lambda_{-k}\mu_{di},\forall~i\neq 0, \forall~k.
\end{align}

\iffalse
\textbf{Claim 2:} $\lambda_1\mu_d=0$; equivalently, $\lambda_1=0=\mu_d$ (since $\lambda_1^2=\mu_d^2$). 

Indeed, assume not, i.e. $\lambda_1=\pm \mu_d\neq 0$. We can deduce from \eqref{eq3: j=1 case} that $\lambda_i=\pm \mu_{di}$. In both cases, we can deduce from \eqref{eq5} that $\lambda_{-k}\lambda_i=0$ for all $i\neq 0$ and $k\in\mathbb{Z}$. In particular, $\lambda_i=0$ for all $i\neq 0$, contradicting to the assumption that $\lambda_1\neq 0$.
\fi

\textbf{Claim.} $\lambda_i=\mu_{di}=0$ for all $i\neq 0$.

First, let us check that $\lambda_i-\mu_{di}=0$ for all $i\neq 0$. Assume this does not hold, then for some $i\neq 0$, we have $\lambda_i-\mu_{di}\neq 0$. By setting $j=i$ in \eqref{eq11: simplied eq1}, we conclude that $\lambda_i^2-\mu_{di}^2=0$. Hence, we have $\lambda_i+\mu_{di}=0$ and thus $\lambda_i=-\mu_{di}$. Then using this relation, we may deduce from \eqref{eq5} and \eqref{eq6} that
$\lambda_i(\lambda_{dik}+\mu_{d^2ik})=\lambda_{-k}\lambda_i$ and $\lambda_i(\lambda_{dik}+\mu_{d^2ik})=-\lambda_{-k}\lambda_i$. Thus
$0=\lambda_{-k}\lambda_i$ for all $k\in\mathbb{Z}$, thus $\lambda_i=0$ and $\mu_{di}=-\lambda_i=0$, contradicting to our assumption that $\lambda_i-\mu_{di}\neq 0$. Hence, we have proved that $\lambda_i=\mu_{di}$ for all $i\neq 0$. Then, it follows from \eqref{eq5} that $\lambda_{-k}\lambda_i=0$ for all $i\neq 0$, thus $\lambda_i=0$ for all $i\neq 0$. This finishes the proof of this claim.

Based on Claim 2 and the fact that $\mu_j=0$ for all $j\not\in d\mathbb{Z}$, we deduce that $a, b\in\mathbb{C}$, then by taking trace on $E(e_1s)=a+bs$, we get $a=0$. By taking $E$ on $E(e_1s)=bs$, we get $bs=E(bs)=bE(s)=0$, i.e. $b=0$, and hence $E(e_1s)=0$. 

This finishes the proof of Lemma \ref{3.5}.
\end{proof}
\begin{remark}\label{remark: abelian assumption is used to show E(s)=0}
Note that in the proof of Lemma \ref{3.5}, the abelian assumption on $Q$ is only needed in Step 1, i.e. proving $E(s)=0$. Once we have $E(s)=0$, then the subsequent Step 2, i.e. the proof of Case 1-3 no longer needs the abelian assumption. This observation would be needed for proving Proposition \ref{prop: version of prop 3.2}. 
\end{remark}

%%%%%%%%%%%%%%%%%%%%%%%%%%%%%%%%%%%%%%%%%%%%%%

Let us classify all invariant von Neumann subalgebras in $L(G_n)$ for even $n$.

\begin{prop}\label{prop: version of prop 3.2}
Let $G_{n}=\mathbb{Z}^n\rtimes SL_n(\mathbb{Z}),n\ge 2$ and n is even. Let $P$ be a $G_{n}$-invariant von Neumann subalgebra in $(L(G_{n}),\tau)$. Then 
\begin{itemize}
\item either $P=L(H)$ for some normal subgroup $H\lhd G_{n}$; or,
\item $P=A_d$ for some $d\geq 1$, where $A_d\subset L(d\mathbb{Z}^n)$ is defined by $A_d=\{x\in L(d\mathbb{Z}^n): \tau(xs)=\tau(xs^{-1}), \forall s\in d\mathbb{Z}^n\}$.
\end{itemize}
\end{prop}

\begin{proof}
By the result of \cite[Theorem A]{aho} and Lemma \ref{3.3}, we know that the center $\mathcal{Z}(P) \subset L(\mathbb{Z}^n \rtimes \{\pm I_{n}\})$. Applying Lemma \ref{3.5} to $Q=\mathcal{Z}(P)$, we may split the proof by considering three cases.

\textbf{Case 1.} $\mathcal{Z}(P) = \mathbb{C}$, i.e. $P$ is a $G_{n}$-invariant subfactor.

Then $P=L(N)$ for some normals subgroup $N$ of $G$ by Proposition \ref{prop: classify_invariant_subfactors}.

\textbf{Case 2.} $\mathcal{Z}(P) = L(d\mathbb{Z}^n)$ for some $d\ge1$.

\textbf{Claim.} $P = L(d\mathbb{Z}^n)$.

Indeed, we just need to check that 
\begin{equation*}
    P\subseteq \mathcal{Z}(P)' \cap L(G_{n}) \subseteq L(\mathbb{Z}^n).
\end{equation*}

The proof is identical to that of proving Case 1 in Proposition \ref{prop: version of prop 3.1} for odd $n$.

\textbf{Case 3.} $\mathcal{Z}(P) = A_{d}$ for some $d\ge1$.

By Lemma \ref{3.7}, we obtain $P\subseteq A_{d}' \cap L(G_{n}) \subset L(\mathbb{Z}^n \rtimes \{\pm I_{n}\})$.

\textbf{Claim.}
 $P \in \{A_{d} ~(d\ge1), L(2\mathbb{Z}^n \rtimes \{\pm I_{n}\}),L(\mathbb{Z}^n \rtimes \{\pm I_{n}\})\}$.
\begin{proof}
 Let  $E: L(\mathbb{Z}^n\rtimes \{\pm{I_{n}}\})\rightarrow P$ be the trace $\tau$-preserving conditional expectation onto $P$. Below, we directly write $g$ for the canonical unitary $u_g$ inside $L(G_{n})$ for simplicity and we use the notation $\langle a, b \rangle$ to mean $\tau(b^*a)$  for any $a, b\in L(G_{n})$. And we denote $s=-I_{n}\in SL_{n}(\mathbb{Z})$.

Note that $L(SL_n(\mathbb{Z}))'\cap L(\mathbb{Z}^n\rtimes SL_n(\mathbb{Z}))=L(\{\pm I_{n}\})$, say by a direct calculation using \cite[Lemma 2.7]{aj}.

Since $P$ is $G_{n}$-invariant, we deduce that $gE(s)g^{-1}=E(gsg^{-1})=E(s)$ for all $g\in SL_n(\mathbb{Z})$. Hence, 
$E(s)\in L(SL_n(\mathbb{Z}))'\cap L(\mathbb{Z}^n\rtimes SL_n(\mathbb{Z}))=L(\{\pm I_{n}\})$. Therefore, we may write $E(s)=\lambda+\mu s$, where $\lambda,\mu\in\mathbb{C}$. Then we get that $E(s)=0$ or $s$ via exactly the same computation as in the proof of Lemma \ref{3.5}.

Since $P\cap L(\mathbb{Z}^n)$ is $SL_n(\mathbb{Z})$-invariant, we may apply \cite[Example 5.9]{wit} to deduce that $P\cap L(\mathbb{Z}^n)=\mathbb{C}$, $L(d_{1}\mathbb{Z}^n)$ or $A_{d_1}$ for some $d_{1}\geq 1$. Note that $P\cap L(\mathbb{Z}^n)=\mathbb{C}$ is impossible since  $\mathbb{C}\neq A_d=\mathcal{Z}(P)\subseteq P\cap L(\mathbb{Z}^n)$. Hence we only need to consider the other two possibilities.

\textbf{Case 3-I}: $P\cap L(\mathbb{Z}^n)=L(d_{1}\mathbb{Z}^n)$, for some $d_1\ge 1$.

We need to further distinguish between the situations $E(s)=s$ and $E(s)=0$, which correspond to whether $s\in P$ or not.

\textbf{Subcase 1: $E(s) = 0$.}

If $d_{1}=1$, then $L(\mathbb{Z}^n)\subseteq P$ and thus $E(v)=v$ for all $v\in\mathbb{Z}^n$. Hence  $E(vs)=vE(s)=0$. Thus, $P=L(\mathbb{Z}^n)$. Hence $P$ is abelian, which  contradicts $\mathcal{Z}(P) = A_{d}$.

Assume that $d_{1}\geq 2$. In view of Remark \ref{remark: abelian assumption is used to show E(s)=0}, we may do a calculation similar to that in Case 2 of Lemma \ref{3.5}  to obtain $P = L(d_{1}\mathbb{Z}^n)$. Hence $P$ is abelian, which  contradicts $\mathcal{Z}(P) = A_{d}$.

\textbf{Subcase 2:} $E(s) = s$, i.e. $s\in P$.

If $d_{1}=1$, we have $L(\mathbb{Z}^n)\subseteq P$ and $s\in P$, which implies that $L(\mathbb{Z}^n) \rtimes \{\pm I_{n}\} \subseteq P$. Since $P\subseteq L(\mathbb{Z}^n) \rtimes \{\pm I_{n}\}$ by hypothesis, we have $P = L(\mathbb{Z}^n) \rtimes \{\pm I_{n}\}$ and consequently $\mathcal{Z}(P) = A_{1}$. Recall that $\mathcal{Z}(P)=A_d$, thus $d = 1 = d_{1}$.

Assume that $d_{1}\ge2$. First, for any $v\in\mathbb{Z}^n\setminus {d_{1}\mathbb{Z}^n}$, we observe that $E(v)=0$.

Indeed, 
\begin{align*}
E(v)\in P\cap L(\mathbb{Z}^n)'=P\cap L(\mathbb{Z}^n)=L(d_1\mathbb{Z}^n).
\end{align*}
Thus, from $\langle v-E(v), E(v)\rangle=0$, we deduce that $\langle E(v), E(v) \rangle=0$, i.e. $E(v)=0$.

Thus, for any $v\in \mathbb{Z}^n$, we have either $E(v)=0$ or $v\in d_1\mathbb{Z}^n$ and in this case $E(v)=v$ since $L(d_1\mathbb{Z}^n)=P\cap L(\mathbb{Z}^n)\subset P$.

We are left to consider that $E(vs)$ for all $v\in \mathbb{Z}^n$. Clearly, we have
\begin{equation*}
E(vs)=E(v)s=\begin{cases} 
vs, & v \in d_1\mathbb{Z}^n \\
0, & v \in \mathbb{Z}^n \setminus d_1\mathbb{Z}^n
\end{cases}
\end{equation*}
To summarize, we obtain $P=L(d_{1}\mathbb{Z}^n\rtimes\{\pm I_{n}
\})$, then $\mathcal{Z}(P)=A_{d_{1}}$. Recall that $\mathcal{Z}(P) = A_{d}$ in Case 3, thus $d=d_1$.

Moreover, take any $v\in\mathbb{Z}^n$, we have $vsv^{-1}\in P$, i.e. $(v\sigma_s(v^{-1}))s\in P$, so $v\sigma_s(v^{-1})\in P$. Notice that if we write $v=(x_{1},x_{2},\dots,x_{n} )^t$, then $v\sigma_s(v^{-1})=2(x_1,x_2,\dots,x_n)^t$. Hence $L(2\mathbb{Z}^n\rtimes \{\pm I_{n}\})\subseteq P$. Combining the above analysis, we deduce that $d_{1}$ can only take the value 1 or 2. Since $d_1\geq 2$, we get that $d_1=2$. 

In other words, in this Case 3-II, $P$ can only be $L(\mathbb{Z}^n \rtimes \{\pm I_{n}\})$ or $L(2\mathbb{Z}^n \rtimes \{\pm I_{n}\})$. It is easy to verify that both $\mathbb{Z}^n \rtimes \{\pm I_{n}\}$ and $2\mathbb{Z}^n \rtimes \{\pm I_{n}\}$ are normal subgroups of $G_{n}$.

\textbf{Case 3-II:} $P\cap L(\mathbb{Z}^n)=A_{d_{1}}$, for some $d_1\ge 1$.

Similarly, we distinguish between the situations $E(s)=0$ and $E(s)=s$.

\textbf{Subcase 1: $E(s) = 0$.}

By Remark \ref{remark: abelian assumption is used to show E(s)=0}, we may do a calculation similar to that Case 3 in Lemma \ref{3.5} to obtain $P = A_{d_{1}}$. Hence $P$ is abelian and $\mathcal{Z}(P) = P =A_{d_{1}}$. If $d=d_{1}$, $P$ is exactly $A_{d_{1}}$. Recall that $\mathcal{Z}(P) = A_{d}$, we also have $d=d_1$.

\textbf{Subcase 2: $E(s) = s$.}

In this subcase, we know that $s\in P$.

First, fix any non-zero vector $v\in d_1\mathbb{Z}^n$, we have $E(v)\in P\cap L(\mathbb{Z}^n)'=P\cap L(\mathbb{Z}^n)=A_{d_1}$. Hence, we can write $E(v)=\sum\limits_{\omega\in d_1\mathbb{Z}^n}\lambda_{\omega}\omega$ with symmetric Fourier coefficients. From $0=\langle v-E(v), P \rangle$, we deduce that $\langle v-\sum\limits_{\omega\in d_1\mathbb{Z}^n}\lambda_{\omega}\omega, \omega_0+\omega_{0}^{-1} \rangle=0$ for all $\omega_0\in d_1\mathbb{Z}^n$. Observe that by taking $\omega_0\not\in\{v, v^{-1}\}$, we can get that $0=\lambda_{\omega_0}+\lambda_{\omega_0^{-1}}=2\lambda_{\omega_0}.$ Similarly, by taking $\omega_0=v$, we deduce that $\lambda_{v}=\lambda_{v^{-1}}=\frac{1}{2}$. Hence $E(v)=\frac{v+v^{-1}}{2}$ for all $v\in d_1\mathbb{Z}^n$.

Second, we check that $E(v)=0$ for all $v\in\mathbb{Z}^n\setminus {d_1\mathbb{Z}^n}$.

Observe that we still have $E(v)\in P\cap L(\mathbb{Z}^n)'=P\cap L(\mathbb{Z}^n)=A_{d_{1}}$. Hence, from the fact that $0=\langle v-E(v), E(v) \rangle$, we deduce that $\langle E(v), E(v)\rangle=0$ for all $v\in \mathbb{Z}^n\setminus {d_1\mathbb{Z}^n}$ since $\langle v, E(v) \rangle=0$ for such a $v$, thus $E(v)=0$ is proved.

Let us compute $E(vs)$ for all $v\in \mathbb{Z}^n$. 
\begin{equation*}
E(vs)=E(v)s=\begin{cases} 
\dfrac{v+v^{-1}}{2}s\in A_{d_{1}} \rtimes \{\pm I_{n}\}, & v \in d_1\mathbb{Z}^n \\
0, & v \in \mathbb{Z}^n \setminus d_1\mathbb{Z}^n
\end{cases}
\end{equation*}
To summarize, we obtain that $P\subseteq A_{d_{1}}\rtimes\{\pm I_{n}
\}$. On the other hand, $P\cap L(\mathbb{Z}^n)=A_{d_{1}}$ and $E(s)=s$ yield $A_{d_{1}} \subseteq P$ and $s\in P$, which in turn gives $A_{d_{1}}\rtimes\{\pm I_{n}\}\subseteq P$, and hence  
$P = A_{d_{1}}\rtimes\{\pm I_{n}\}$.

Moreover, take any $v\in\mathbb{Z}^n$, we have $vsv^{-1}\in P$, i.e. $(v\sigma_s(v^{-1}))s\in P$, so $v\sigma_s(v^{-1})\in P$. Notice that if we write $v=(x_{1},x_{2},\dots,x_{n} )^t$, then $v\sigma_s(v^{-1})=2(x_1,x_2,\dots,x_n)^t$. Hence $L(2\mathbb{Z}^n\rtimes \{\pm I_{n}\})\subseteq P$. It is clear $P = A_{d_{1}}\rtimes\{\pm I_{n}\})$ contradicts $L(2\mathbb{Z}^n\rtimes \{\pm I_{n}\})\subseteq P$.
\end{proof}

Through the analysis of the  above three cases, we conclude that $P$ appears in the following list of subaglebras: 
\begin{itemize}
\item $\mathbb{C}, L(d\mathbb{Z}^n)(d\ge 1), L(2\mathbb{Z}^n \rtimes \{\pm I_{n}\})$, $L(\mathbb{Z}^n \rtimes \{\pm I_{n}\})$,
\item $L(H) ~\text{(for some normal subgroup $H\lhd G_{n}$)}$,
$A_{d} ~(d\ge1)$.
\end{itemize}

Except the case $P = A_{d}$, all others can be expressed in the form $P=L(H)$ for some normal subgroup $H\lhd G_{n}$. 
\end{proof}

Finally, here is the proof of Corollary \ref{cor}.

\begin{proof}[Proof of Corollary \ref{cor}]
    Let $P$ be a $G_n$-invariant von Neumann subalgebra in $L(G_n)$ with the Haagerup property. Then $P=A_n$ for some $n\geq 0$ or $L(H)$ for some normal subgroup $H\lhd G$ with the Haagerup property by Theorem \ref{thm}. By Lemma \ref{lem: Haagerup radical in G_n}, we know that $H\subseteq \mathbb{Z}^n\rtimes \{\pm I_n\}$, where $I_n$ denotes the identity matrix in $SL_n(\mathbb{Z})$. Therefore, $P\subseteq L(\mathbb{Z}^n\rtimes \{\pm I_n\})$ in both cases.
Notice that $L(\mathbb{Z}^n\rtimes \{\pm I_n\})$ is clearly $G_n$-invariant and has Haagerup property and hence it is the maximal one with these properties.
\end{proof}


\begin{thebibliography}{9}


\bib{ab}{article}{
   author={Alekseev, V.},
   author={Brugger, R.},
   title={A rigidity result for normalized subfactors},
   journal={J. Operator Theory},
   volume={86},
   date={2021},
   number={1},
   pages={3--15},}

\bib{adjs}{article}{
author={Amrutam, T.},
author={Dudko, A.},
author={Jiang, Y.},
author={Skalski, A.},
title={Invariant subalgebras rigidity for von Neumann algebras of groups arising as certain semidirect products},
date={2025},
status={arXiv: 2507.12824},
}

\bib{aho}{article}{
author={Amrutam, T.},
author={Hartman, Y.},
author={Oppelmayer, H.},
title={On the amenable subalgebras of group von Neumann algebras},
 journal={J. Funct. Anal.},
   volume={288},
   date={2025},
   number={2},
   pages={Paper No. 110718, 20 pp},
}
\bib{aj}{article}{
   author={Amrutam, T.},
   author={Jiang, Y.},
   title={On invariant von Neumann subalgebras rigidity property},
   journal={J. Funct. Anal.},
   volume={284},
   date={2023},
   number={5},
   pages={Paper No. 109804},}


\bib{bbhp}{article}{
   author={Bader, Uri},
   author={Boutonnet, R\'{e}mi},
   author={Houdayer, Cyril},
   author={Peterson, Jesse},
   title={Charmenability of arithmetic groups of product type},
   journal={Invent. Math.},
   volume={229},
   date={2022},
   number={3},
   pages={929--985},}


\bib{bekka_invent}{article}{
   author={Bekka, Bachir},
   title={Operator-algebraic superridigity for ${\rm SL}_n(\Bbb Z)$, $n\geq
   3$},
   journal={Invent. Math.},
   volume={169},
   date={2007},
   number={2},
   pages={401--425},}
   
\bib{bd}{book}{
   author={Bekka, B.},
   author={de la Harpe, P.},
   title={Unitary representations of groups, duals, and characters},
   series={Mathematical Surveys and Monographs},
   volume={250},
   publisher={American Mathematical Society, Providence, RI},
   date={2020},
   pages={xi+474},}
   
   
   \bib{bdv}{book}{
   author={Bekka, Bachir},
   author={de la Harpe, Pierre},
   author={Valette, Alain},
   title={Kazhdan's property (T)},
   series={New Mathematical Monographs},
   volume={11},
   publisher={Cambridge University Press, Cambridge},
   date={2008},
   pages={xiv+472},}

\bib{bh}{article}{
   author={Boutonnet, R\'{e}mi},
   author={Houdayer, Cyril},
   title={Stationary characters on lattices of semisimple Lie groups},
   journal={Publ. Math. Inst. Hautes \'{E}tudes Sci.},
   volume={133},
   date={2021},
   pages={1--46},}

   


\bib{burger}{article}{
   author={Burger, M.},
   title={Kazhdan constants for {${\rm SL}(3,{\bf Z})$}},
   journal={J. Reine Angew. Math.},
   volume={413},
   date={1991},
   number={10},
   pages={36--67},}


\bib{cd}{article}{
   author={Chifan, I.},
   author={Das, S.},
   title={Rigidity results for von Neumann algebras arising from mixing
   extensions of profinite actions of groups on probability spaces},
   journal={Math. Ann.},
   volume={378},
   date={2020},
   number={3-4},
   pages={907--950},}

\bib{cds}{article}{
   author={Chifan, I.},
   author={Das, S.},
   author={Sun, B.},
   title={Invariant subalgebras of von Neumann algebras arising from
   negatively curved groups},
   journal={J. Funct. Anal.},
   volume={285},
   date={2023},
   number={9},
   pages={Paper No. 110098},}


\bib{cios}{article}{
   author={Chifan, Ionu\c{t}},
   author={Ioana, Adrian},
   author={Osin, Denis},
   author={Sun, Bin},
   title={Wreath-like products of groups and their von Neumann algebras I:
   $\rm W^\ast $-superrigidity},
   journal={Ann. of Math. (2)},
   volume={198},
   date={2023},
   number={3},
   pages={1261--1303},}


\bib{cs}{article}{
   author={Chifan, I.},
   author={Sinclair, T.},
   title={On the structural theory of ${\rm II}_1$ factors of negatively
   curved groups},
   journal={Ann. Sci. \'{E}c. Norm. Sup\'{e}r. (4)},
   volume={46},
   date={2013},
   number={1},
   pages={1--33},}

\bib{connes_T}{article}{
   author={Connes, A.},
   title={A factor of type ${\rm II}_{1}$ with countable fundamental
   group},
   journal={J. Operator Theory},
   volume={4},
   date={1980},
   number={1},
   pages={151--153},}
   
   \bib{cp}{article}{
   author={Creutz, Darren},
   author={Peterson, Jesse},
   title={Character rigidity for lattices and commensurators},
   journal={Amer. J. Math.},
   volume={146},
   date={2024},
   number={3},
   pages={687--711},}



   \bib{dp}{article}{
   author={Das, Sayan},
   author={Peterson, Jesse},
   title={Poisson boundaries of ${\rm II}_1$ factors},
   journal={Compos. Math.},
   volume={158},
   date={2022},
   number={8},
   pages={1746--1776},}


   \bib{dgghl}{article}{
author={Dogon, A.},
author={Glasner, M.},
author={Gorfine, Y.},
author={Hanany, L.},
author={Levit, A.},
title={Non-uniform higher-rank lattices are character rigid},
status={arXiv: 2507.21862},
date={2025},}
   
   \bib{dj}{article}{
   author={Dudko, A.},
   author={Jiang, Y.},
   title={A character approach to the ISR property},
   journal={arXiv preprint arXiv:2410.14517},
   date={2024},}

\bib{houdayer_icm}{article}{
   author={Houdayer, Cyril},
   title={Noncommutative ergodic theory of higher rank lattices},
   conference={
      title={ICM---International Congress of Mathematicians. Vol. 4.
      Sections 5--8},
   },
   book={
      publisher={EMS Press, Berlin},
   },
   date={2023},
   pages={3202--3223},}

 \bib{dl}{book}{
   author={Kerr, D.},
   author={Li, H.},
   title={Ergodic theory. Independence and dichotomies},
   series={Springer Monographs in Mathematics},
   publisher={Springer, Cham},
   date={2016},
   pages={xxxiv+431},}
   
   

  
   

   
   \bib{ioana_icm}{article}{
   author={Ioana, A.},
   title={Rigidity for von Neumann algebras},
   conference={
      title={Proceedings of the International Congress of
      Mathematicians---Rio de Janeiro 2018. Vol. III. Invited lectures},
   },
   book={
      publisher={World Sci. Publ., Hackensack, NJ},
   },
   date={2018},
   pages={1639--1672},}
   

\bib{jiang_jfa}{article}{
   author={Jiang, Yongle},
   title={Continuous orbit equivalence rigidity for left-right wreath
   product actions},
   journal={J. Funct. Anal.},
   volume={285},
   date={2023},
   number={2},
   pages={Paper No. 109942, 35 pages},}
   
   
   
   \bib{jiangliu}{article}{
   author={Jiang, Y.},
   author={Liu, R.},
   title={On invariant subalgebras when the ISR property fails},
   status={accepted to J. Operator Theory},
 year={2023},
   }

   \bib{jiangskalski}{article}{
   author={Jiang, Y.},
   author={Skalski, A.},
   title={Maximal subgroups and von Neumann subalgebras with the Haagerup
   property},
   journal={Groups Geom. Dyn.},
   volume={15},
   date={2021},
   number={3},
   pages={849--892},}


\bib{jz}{article}{
author={Jiang, Y.},
author={Zhou, X.},
title={An example of an infinite amenable group with the ISR property},
journal={Math. Z.},
   volume={307},
   date={2024},
   number={2},
   pages={Paper No. 23, 11 pp},
}


\bib{jones_10}{article}{
   author={Jones, V. F. R.},
   title={Ten problems},
   conference={
      title={Mathematics: frontiers and perspectives},
   },
   book={
      publisher={Amer. Math. Soc., Providence, RI},
   },
   isbn={0-8218-2070-2},
   date={2000},
   pages={79--91},}

   

\bib{kp}{article}{
   author={Kalantar, M.},
   author={Panagopoulos, N.},
   title={On invariant subalgebras of group and von Neumann algebras},
   journal={Ergodic Theory Dynam. Systems},
   volume={43},
   date={2023},
   number={10},
   pages={3341--3353},}

\bib{kazhdan_T}{article}{
   author={Ka\v{z}dan, D. A.},
   title={On the connection of the dual space of a group with the structure
   of its closed subgroups},
   language={Russian},
   journal={Funkcional. Anal. i Prilo\v{z}en.},
   volume={1},
   date={1967},
   pages={71--74},}

   
   \bib{Margulis_book}{book}{
   author={Margulis, G. A.},
   title={Discrete subgroups of semisimple Lie groups},
   series={Ergebnisse der Mathematik und ihrer Grenzgebiete (3) [Results in
   Mathematics and Related Areas (3)]},
   volume={17},
   publisher={Springer-Verlag, Berlin},
   date={1991},
   pages={x+388},}
   
   
\bib{Morris_book}{book}{
   author={Morris, Dave W.},
   title={Introduction to arithmetic groups},
   publisher={Deductive Press, [place of publication not identified]},
   date={2015},
   pages={xii+475},}   

   \bib{ozawa_solid}{article}{
   author={Ozawa, N.},
   title={Solid von Neumann algebras},
   journal={Acta Math.},
   volume={192},
   date={2004},
   number={1},
   pages={111--117},}
   

\bib{packer}{article}{
author={Packer, Judith A.},
title={On the embedding of subalgebras corresponding to quotient
              actions in group-measure factors},
journal={Pacific J. Math.},
   volume={119},
   date={1985},
   number={2},
   pages={407--443},
}

\bib{peterson_online}{article}{
author={Peterson, J.},
title={Character rigidity for lattices in higher-rank groups}
date={2014},
status={unpublished note, available at the link: \url{https://math.vanderbilt.edu/peters10/rigidity.pdf}},
}


\bib{pt}{article}{
   author={Peterson, Jesse},
   author={Thom, Andreas},
   title={Character rigidity for special linear groups},
   journal={J. Reine Angew. Math.},
   volume={716},
   date={2016},
   pages={207--228},}
  

\bib{popa_T1}{article}{
   author={Popa, S.},
   title={Strong rigidity of $\rm II_1$ factors arising from malleable
   actions of $w$-rigid groups. I},
   journal={Invent. Math.},
   volume={165},
   date={2006},
   number={2},
   pages={369--408},}




   \bib{popa_T2}{article}{
   author={Popa, S.},
   title={Strong rigidity of $\rm II_1$ factors arising from malleable
   actions of $w$-rigid groups. II},
   journal={Invent. Math.},
   volume={165},
   date={2006},
   number={2},
   pages={409--451},}

   \bib{popa_Tc}{article}{
   author={Popa, S.},
   title={Cocycle and orbit equivalence superrigidity for malleable actions
   of $w$-rigid groups},
   journal={Invent. Math.},
   volume={170},
   date={2007},
   number={2},
   pages={243--295},}

\bib{popa_icm}{article}{
   author={Popa, S.},
   title={Deformation and rigidity for group actions and von Neumann
   algebras},
   conference={
      title={International Congress of Mathematicians. Vol. I},
   },
   book={
      publisher={Eur. Math. Soc., Z\"{u}rich},
   },
   date={2007},
   pages={445--477},}

   
   \bib{vaes_icm}{article}{
   author={Vaes, S.},
   title={Rigidity for von Neumann algebras and their invariants},
   conference={
      title={Proceedings of the International Congress of Mathematicians.
      Volume III},
   },
   book={
      publisher={Hindustan Book Agency, New Delhi},
   },
   date={2010},
   pages={1624--1650},}

\bib{valette}{article}{
   author={Valette, Alain},
   title={Maximal Haagerup subgroups in $\Bbb Z^{n+1}\rtimes_{\rho_n} {\rm
   GL}_2(\Bbb Z)$},
   journal={Studia Math.},
   volume={275},
   date={2024},
   number={3},
   pages={263--283},}

\bib{wit}{article}{
   author={Witte, D.},
   title={Measurable quotients of unipotent translations on homogeneous
   spaces},
   journal={Trans. Amer. Math. Soc.},
   volume={345},
   date={1994},
   number={2},
   pages={577--594},}


   
\end{thebibliography}
\end{document}